\documentclass[12pt,a4paper]{amsart}
\usepackage[utf8]{inputenc}

\title[Rowmotion and the Octahedron Recurrence]{Birational Rowmotion and the Octahedron Recurrence}
\author{Joseph Johnson}
\address{North Carolina State University, Raleigh, NC 27695}
\email{jwjohns5@ncsu.edu}
\author{Ricky Ini Liu}
\address{University of Washington, Seattle, WA 98195}
\email{riliu@uw.edu}
\thanks{R.\ I.\ Liu was partially supported by National Science Foundation grants DMS-1700302 and CCF-1900460.}
\date{\today}

\usepackage{amssymb}
\usepackage{amsfonts}
\usepackage{amsthm}
\usepackage{amsmath}
\usepackage{amscd}
\usepackage{t1enc}
\usepackage[mathscr]{eucal}
\usepackage{indentfirst}
\usepackage{graphicx}
\usepackage{graphics}
\usepackage{pict2e}
\usepackage{epic}
\numberwithin{equation}{section}
\usepackage[margin=2.9cm]{geometry}
\usepackage{epstopdf} 
\usepackage{tikz}
\tikzstyle{w}=[circle, draw, fill=black, inner sep=0pt, minimum width=4pt]
\usepackage{transparent}
\usepackage{wasysym}
\usepackage{enumerate}
\usepackage{hyperref}

\usepackage{mathdots}
\usepackage{todonotes}

\usepackage{caption}

 \colorlet{myGreen}{green!50!gray!120!}

\theoremstyle{plain}
\newtheorem{Th}{Theorem}[section]
\newtheorem{Lemma}[Th]{Lemma}
\newtheorem{Cor}[Th]{Corollary}
\newtheorem{Prop}[Th]{Proposition}

 \theoremstyle{definition}
\newtheorem{Def}[Th]{Definition}

\newtheorem{?}[Th]{Problem}
\newtheorem{Ex}[Th]{Example}

\newcommand{\op}{\mathcal{O}(P)}
\newcommand{\cp}{\mathcal{C}(P)}
\newcommand{\ST}{ST}

\makeatletter
\newcommand\bigDiamond{\mathop{\mathpalette\bigDi@mond\relax}}
\newcommand\bigDi@mond[2]{%
  \vcenter{\hbox{\m@th
    \scalebox{\ifx#1\displaystyle 2\else1.2\fi}{$#1\Diamond$}%
  }}%
}

\begin{document}

\maketitle

\begin{abstract}
We use the octahedron recurrence to give a simplified statement and proof of a formula for iterated birational rowmotion on a product of two chains, first described by Musiker and Roby. Using this, we show that weights of certain chains in rectangles shift in a predictable way under the action of rowmotion. We then define generalized Stanley-Thomas words whose cyclic rotation uniquely determines birational rowmotion on the product of two chains. We also discuss the relationship between rowmotion and birational RSK and give a birational analogue of Greene's theorem in this setting.
\end{abstract}

\section{Introduction}

For any poset $P$, \emph{(combinatorial) rowmotion} is the action on the set of order ideals of $P$ that sends $I$ to the ideal generated by the minimal elements of $P \setminus I$. This action is well studied in the dynamical algebraic combinatorics literature; for background on rowmotion, see \cite{strikerwilliams}. 
On certain classes of posets (triangles, skeletal posets, rectangles, root posets, and others), rowmotion has a surprisingly small period, and it also sometimes exhibits other interesting phenomena such as homomesy and cyclic sieving: see \cite{einsteinpropp1,josephroby1,musikerroby,propproby,thomaswilliams}.

Rowmotion also has a description in terms of local, involutive transformations called \emph{toggles} \cite{cameronfonderflaass}. Reinterpreting these toggles as acting on lattice points in $\mathbb R^n$, one can lift toggles and hence rowmotion to piecewise-linear maps. One can then lift these further to the birational realm by replacing $\max$ with addition, addition with multiplication, and subtraction with division \cite{einsteinpropp2,einsteinpropp1}. Surprisingly, many results from the combinatorial level remain true on the birational level. For instance, for some posets, the period of birational rowmotion remains small, even though a priori it need not even be finite \cite{grinbergroby2,grinbergroby1}. 

The main poset of interest in this paper is the product of two chains, called the \emph{rectangle poset}. In \cite{grinbergroby2}, Grinberg and Roby show that birational rowmotion on the $r \times s$ rectangle has the same order as combinatorial rowmotion, $r+s$.  Musiker and Roby~\cite{musikerroby} then give an explicit combinatorial formula for all powers of birational rowmotion on rectangles in terms of nonintersecting lattice paths. However, their impressive formula is notationally dense, and their proof requires a rather intricate bijection on lattice paths. The main result of this paper is a simplified statement and proof of Musiker and Roby's iterated birational rowmotion formula for rectangles. Our proof relies mainly on the connections between rowmotion, the octahedron recurrence, the solid minors of a matrix, and nonintersecting lattice paths via the Lindstr\"om-Gessel-Viennot Lemma.

We also touch upon a number of topics related to this work. For instance, associated to any antichain of the $r \times s$ rectangle poset is a certain $0/1$-sequence of length $r+s$ called the \emph{Stanley-Thomas word}. Rowmotion on order ideals equivariantly induces a rowmotion action on antichains, which cyclically shifts the Stanley-Thomas word \cite{propproby,stanley1}. Previously, the Stanley-Thomas word has been defined in the birational realm to prove homomesy and cyclic sieving results \cite{josephroby1}, but this word alone is not enough to uniquely define a general labeling of a rectangle. In this article, we define \emph{generalized Stanley-Thomas words} in terms of certain sums of weights of chains and show that birational rowmotion is the unique function that cyclically shifts all of them. 





Finally, we discuss the relationship between the birational version of the \emph{Robinson-Schensted-Knuth (RSK) correspondence} \cite{danilovkoshevoy, noumiyamada} and rowmotion by defining birational RSK in terms of toggles. We use the iterated birational rowmotion formula to show that this definition satisfies a birational version of Greene's Theorem and also compare it to existing constructions in the literature.

\subsection*{Road map of the paper} In Section~\ref{section:background} we review background on birational rowmotion, the octahedron recurrence, and the relationship between the two. In Section~\ref{section:rowmotiononrectangles} we state and prove our main result, the iterated birational rowmotion formula on rectangles, using the Lindstr\"om-Gessel-Viennot Lemma and the octahedron recurrence. Using this framework we prove a chain shifting lemma in Section~\ref{section:chainshifting} and define generalized Stanley-Thomas words. Finally in Section~\ref{section:greenestheorem} we define birational RSK and prove a birational analogue of Greene's theorem, which we then use to show that the cyclic shifting of the generalized Stanley-Thomas words uniquely determines birational rowmotion.

\section{Background} \label{section:background}

\subsection{Posets and rectangles}
\label{subsection:posets}

We first review some basic terminology about posets. Typically we represent a finite poset $P=(P, \preceq)$ by its \emph{Hasse diagram}, a directed graph with vertex set $P$ and edges $x \to y$ for $x \lessdot y$. Since all edges are directed upward in the Hasse diagram, we omit the direction in figures.

\begin{Def}
Let $P$ be a poset. A \emph{chain} in $P$ is a sequence $p_1 \preceq p_2 \preceq \dots \preceq p_k$. An \emph{antichain} in $P$ is a set $A \subseteq P$ such that for any distinct $p,q \in A$, neither $p \preceq q$ nor $q \preceq p$.
\end{Def}

The antichains are related to the order ideals (and order filters) of a poset.

\begin{Def} Let $P$ be a finite poset.
\begin{enumerate}
    \item An \emph{order ideal} of $P$ is a set $I \subseteq P$ such that if $p,q \in P$ such that $p \preceq q$ and $q \in I$, then $p \in I$. 
    \item An \emph{order filter} of $P$ is a set $F \subseteq P$ such that if $p,q \in P$ such that $p \preceq q$ and $p \in I$, then $q \in I$. 
    \item An \emph{interval} of $P$ is a subset of the form $[p,q]=\{ x \mid p \preceq x \preceq q \}$.
\end{enumerate}
\end{Def}

Let $[r]$ be the chain with $r$ elements $1 < 2 < \cdots < r$. Of particular interest is the \emph{rectangle poset} given by the Cartesian product of two chains $R = [r] \times [s]$. We will distinguish between a rectangle poset and general posets by using $R$ exclusively for the rectangle. 

\begin{Def}
Let $R = [r] \times [s]$ be a rectangle poset.
\begin{itemize}
    \item For fixed $i$, the $i$th \emph{row} of $R$ is the set of all elements in $R$ of the form $(i,j)$. 
    \item For fixed $j$, the $j$th \emph{column} of $R$ is the set of all elements in $R$ of the form $(i,j)$. 
    \item The $k$th \emph{rank} of $R$ is the set of all elements $(i,j)$ such that $i+j=k$.
    \item The $k$th \emph{file} of $R$ is the set of all elements $(i,j)$ such that $j-i=k$.
\end{itemize}
\end{Def}
We will typically draw rectangles oriented as in Figure~\ref{squareposet}, so that rows run southwest to northeast, columns run southeast to northwest, ranks are aligned horizontally, and files are aligned vertically. (As a word of caution, note that the minimum element has rank $2$.)

\begin{figure}
\captionsetup{justification=centering}
\begin{tikzpicture}[rotate=45,scale=1.1]
\draw (0,0)--(2,0);
\draw (0,1)--(2,1);
\draw (0,0)--(0,1);
\draw (1,0)--(1,1);
\draw (2,0)--(2,1);

\draw [fill=black] (0,0) circle [radius=0.1cm];
\draw [fill=black] (1,0) circle [radius=0.1cm];
\draw [fill=black] (2,0) circle [radius=0.1cm];

\draw [fill=black] (0,1) circle [radius=0.1cm];
\draw [fill=black] (1,1) circle [radius=0.1cm];
\draw [fill=black] (2,1) circle [radius=0.1cm];

\node at (0.45,-0.45) {(1,1)};
\node at (1.45,-0.45) {(1,2)};
\node at (2.4,-0.4) {(1,3)};

\node at (0.45,0.55) {(2,1)};
\node at (1.45,0.55) {(2,2)};
\node at (2.45,0.55) {(2,3)};
\end{tikzpicture}

\caption{The rectangle $[2] \times [3]$.}
 \label{squareposet}
\end{figure}
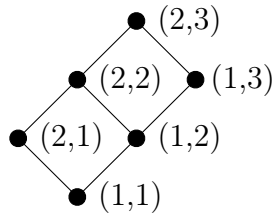

\subsection{Rowmotion}

An important object of study in dynamical algebraic combinatorics is a certain dynamical process on order ideals called combinatorial rowmotion.

\begin{Def}
Let $I$ be an order ideal of $P$. \textit{Combinatorial rowmotion} is the map $\rho$ that sends $I$ to the order ideal generated by the minimal elements of $P \setminus I$.
\end{Def}

Cameron and Fon-der-Flaass \cite{cameronfonderflaass} give another description of rowmotion in terms of \textit{combinatorial toggles}. Let $J(P)$ denote the set of order ideals of $P$. For each $p \in P$, we associate a \emph{toggle map} $t_p \colon J(P) \to J(P)$ by
\[t_p(I) = \begin{cases} I \cup \{p\} &\text{ if } p \not\in I \text{ and } I \cup \{p\} \in J(P), \\
I \setminus \{p\} &\text{ if } p \in I \text{ and } I \setminus \{p\} \in J(P), \\
I &\text{ otherwise.} \end{cases}\]
Combinatorial rowmotion can then be defined as the composition of toggles on $P$ in the order of a linear extension $\mathcal L \colon P \to [n]$ (where $n=|P|$) from top to bottom, that is,
\[\rho = t_{\mathcal{L}^{-1}(1)} \circ \dots \circ t_{\mathcal{L}^{-1}(n)}.\]
We note that $t_p \circ t_q = t_q \circ t_p$ if and only if $p$ and $q$ do not form a cover relation. Consequently this description is independent of the choice of linear extension.

In \cite{stanley1}, Stanley gives a bijection (also discovered independently by Thomas) between the order ideals of $R = [r] \times [s]$ and 0/1-sequences with $r$ 0's and $s$ 1's.

\begin{Def}
Let $A$ be an antichain of $R = [r] \times [s]$. The \emph{Stanley-Thomas word} of $A$ is $w(A) = (w_1, \dots, w_{r+s})$, where

\[w_i = \begin{cases} 1 &\text{ if } 1\leq i \leq r \text{ and } A \text{ has an element in row } i, \\ 
1 &\text{ if } r+1 \leq i \leq r+s \text{ and } A \text{ has no element in column } i-r, \\ 
0 &\text{ otherwise.} \end{cases}\]
\end{Def}

Rowmotion performs a cyclic shift of the Stanley-Thomas word \cite{propproby}. Since the map $w$ is a bijection, this gives a simple proof of the following result, originally proved by Fon-der-Flaass \cite{fonderflaass}.

\begin{Th}[Fon-Der-Flaass]
The order of $\rho$ on $R=[r] \times [s]$ is $r+s$.
\end{Th}

Note that $r+s$ is one more than the number of ranks of $R$. This is the smallest possible order that rowmotion can have on a graded poset: iteratively applying rowmotion to the empty order ideal simply adds one rank of elements at a time to the order ideal until arriving at the entire poset, which is then mapped back to the empty order ideal.

\subsection{Piecewise-Linear Rowmotion}

In \cite{stanley2}, Stanley defines two polytopes related to a finite poset.

\begin{Def}
Let $P$ be a finite poset. 

\begin{enumerate}
    \item The \emph{order polytope} $\op \subseteq \mathbb{R}^{P}$ is defined by the inequalities $0 \leq x_p \leq 1$ for all $p \in P$, and $x_p \leq x_q$ for all $p,q \in P$ satisfying $p \preceq q$.
    \item The \emph{chain polytope} $\cp \subseteq \mathbb{R}^{P}$ is defined by the inequalities $x_p \geq 0$ for all $p \in P$, and $\sum\limits_{p \in C} x_p \leq 1$ for all (maximal) chains $C \subseteq P$.
\end{enumerate}
\end{Def}
The vertices of $\op$ and $\cp$ are the indicator vectors of the order filters and antichains of $P$, respectively. We may also identify the vertices of $\op$ with the order ideals of $P$ by complementation.

In \cite{stanley2}, Stanley defines a piecewise-linear, continuous, volume-preserving bijection between $\op$ and $\cp$ called the \emph{transfer map}, defined as:
\[\phi(x)_p = x_p - \max\limits_{q \lessdot p} x_q\]
where we interpret an empty $\max$ as $0$.
The inverse of this map is given by
\[\phi^{-1}(x)_p = \max\limits_{c_1 < c_2 < \dots < c_k = x} \left( \sum\limits_{i=1}^k x_{c_i} \right).\]
The transfer map can be thought of as a piecewise-linear extension of the map that sends an order filter to its minimal elements.

Combinatorial rowmotion permutes order ideals and so can also be thought of as a permutation of the vertices of $\op$. As is the case with the transfer map, there is a natural piecewise-linearization of this bijection from \cite{einsteinpropp1}.

\begin{Def}
The \textit{piecewise-linear toggle} on the order polytope $\op$ corresponding to an element $p \in P$ is the map $t_p\colon \op \to \op$ that changes the $p$th coordinate by
\[x_p \mapsto \min\limits_{q \gtrdot p} x_q + \max\limits_{q \lessdot p} x_q - x_p\]
and fixes all other coordinates, where we interpret an empty $\min$ as $1$ and an empty $\max$ as $0$.
\end{Def}
(We will abuse notation and use the same symbol for combinatorial toggles and piecewise-linear toggles.)
Note that $t_p$ only depends on the coordinates in the neighborhood of $p \in P$ in the Hasse diagram. Consequently for $p,q \in P$, $t_p \circ t_q = t_q \circ t_p$ if and only if neither $p \lessdot q$ nor $p \gtrdot q$ as in the combinatorial realm.

We then define \textit{piecewise-linear rowmotion} $\rho\colon \op \to \op$ using piecewise-linear toggles by
\[\rho = t_{\mathcal{L}^{-1}(1)} \circ \cdots \circ t_{\mathcal{L}^{-1}(|P|)},\]
where $\mathcal{L}$ is any linear extension of $P$.

\subsection{Birational Rowmotion}

Piecewise-linear rowmotion can be lifted even further to a birational analogue that uses addition, multiplication, and division in place of $\max$, addition, and subtraction, respectively. This lifting process is called \emph{detropicalization}. (See \cite{einsteinpropp2,einsteinpropp1,noumiyamada} for more detailed discussion.) The functions resulting from detropicalization are generally \emph{subtraction-free} and therefore well-defined on positive labelings of $P$.

As an example, the \emph{birational transfer map} $\phi$ acts on positive labelings $x \in \mathbb R^P_{>0}$ via coordinate functions
\[\phi(x)_p = \frac{x_p}{\sum\limits_{q \lessdot p} x_q}\]
where we interpret an empty sum as $1$.

Since $\min(a,b) = -\max(-a,-b)$, detropicalizing $\min$ yields the \textit{parallel sum}~$\parallel$ defined by
\[a \parallel b = \frac{1}{\frac{1}{a} + \frac{1}{b}} = \frac{ab}{a+b}.\]
Parallel sum is associative and commutative. For a finite set $S \subseteq \mathbb{R}_{>0}$, we denote the parallel sum of all elements in $S$ by $\sideset{}{^\parallel} \sum\limits_{s \in S} s$.

\begin{Def}
The \textit{birational toggle} on $\mathbb{R}^{P}_{>0}$ corresponding to an element $p \in P$ is the birational map $t_p\colon \mathbb{R}^{P}_{>0} \to \mathbb{R}^{P}_{>0}$ that changes the $p$th coordinate by
\[x_p \mapsto \left( \sideset{}{^\parallel} \sum\limits_{q \gtrdot p} x_q \right) \left( \sum\limits_{q \lessdot p} x_q \right) \cdot \frac{1}{x_p}\]
and fixes all other coordinates, where we interpret an empty sum or empty parallel sum as $1$.
\end{Def}
(We again abuse notation by using the same notation for birational toggles as piecewise-linear toggles.) Similarly to piecewise-linear rowmotion, we define \textit{birational rowmotion} as
\[\rho = t_{\mathcal{L}^{-1}(1)} \circ \cdots \circ t_{\mathcal{L}^{-1}(|P|)}\]
for any linear extension $\mathcal{L}$ of $P$. From the birational setting, we can obtain the piecewise linear analogue via a \textit{valuation}: see, for instance, \cite{einsteinpropp2,einsteinpropp1}. 

For any poset, one can compute birational rowmotion in terms of the dual transfer map.

\begin{Def}
Let $P$ be a poset and $x \in \mathbb{R}_{>0}^{P}$ be a labeling. The \emph{dual transfer map} is the birational function $\phi^*\colon\mathbb{R}_{>0}^{P} \to \mathbb{R}_{>0}^{P}$ with coordinate functions
\[\phi^*(x)_p = \frac{x_p}{\sum\limits_{q \gtrdot p} x_q}\]
for all $p \in P$, where we interpret an empty sum as $1$.
\end{Def}

In other words, $\phi^*$ acts on labelings of $P$ in the same way that $\phi$ acts on the associated labeling of the dual of $P$. 

The following lemma is due to Einstein and Propp (see for instance \cite{einsteinpropp2,josephroby2}). We include a short proof for completeness.

\begin{Lemma} \label{lemma:dualtransfer}
Let $P$ be a finite poset. Then for any $x \in \mathbb{R}_{>0}^{P}$ and $p \in P$,
\[ \rho \circ \phi^{-1}(x)_p = \frac{1}{\left(\phi^*\right)^{-1}(x)_p}.\]
\end{Lemma}
\begin{proof}
Let $y = \phi^{-1}(x)$ and $z = (\phi^*)^{-1}(x)$. For any $p \in P$, suppose that $\rho(y)_q = \frac{1}{z_q}$ for all $q > p$. When applying the toggle at $p$ in the computation of $\rho(y)$, the elements at or below $p$ are labeled as in $y$ while those above $p$ are labeled as in $\rho(y)$. Hence by the definition of $t_p$,
\[
    \rho(y)_p =  \left( \sideset{}{^\parallel} \sum\limits_{q \gtrdot p} \frac{1}{z_q} \right) \left(\sum\limits_{q \lessdot p} y_q \right) \cdot \frac{1}{y_p}  = \frac{1}{\sum\limits_{q \gtrdot p}z_q} \cdot \frac{1}{x_p} = \frac{1}{z_p}.
\]
The result follows easily by induction starting at the top of $P$.
\end{proof}

Computing iterated applications of rowmotion is more difficult. However, one can sometimes prove results about rowmotion using as associated combinatorial construction instead. As mentioned previously, one such example is the Stanley-Thomas word, which lifts to the birational level.

\begin{Def}
Let $R = [r] \times [s]$. The \emph{birational Stanley-Thomas word} $w$ is the word of length $r+s$ defined by
\[w_i = \begin{cases} \prod\limits_{j=1}^s x_{ij} &\text{ if } 1\leq i \leq r, \\ \prod\limits_{j=1}^r x_{j,i-r}^{-1} &\text{ if } r+1\leq i \leq r+s.\end{cases}\]
\end{Def}

In \cite{josephroby1} it is shown that $\phi \circ \rho \circ \phi^{-1}$ cyclically shifts the birational Stanley-Thomas word. Though this is not sufficient to show that the order of birational rowmotion on the product of two chains has finite order, it can be used to extend other results (such as instances of homomesy) to the birational level. In Section \ref{section:chainshifting} we define a collection of \emph{generalized Stanley-Thomas words}, and in Section \ref{section:greenestheorem} we show that the cyclic rotation of these words uniquely determines rowmotion.

\subsection{Dodgson Condensation and the Octahedron Recurrence} \label{sec:octahedron}

Given a matrix $A=(a_{ij})_{i,j=1}^{n}$, let $A_{ij}^{(k)}$ denote the $k \times k$ submatrix of $A$ formed by the intersection of rows $i$ through $i+k-1$ and columns $j$ through $j+k-1$ whenever $1 \leq i,j \leq n-k+1$. These minors satisfy the following algebraic relation known as the \emph{Desnanot-Jacobi identity}, which forms the basis of a recursive algorithm for computing the determinant of a matrix called \emph{Dodgson condensation}.

\begin{Prop} \label{prop:dodgson}
For $k \geq 0$,
\[\det(A_{ij}^{(k+1)}) \det(A_{i+1,j+1}^{(k-1)})=\det(A_{ij}^{(k)})\det(A_{i+1,j+1}^{(k)})-\det(A_{i,j+1}^{(k)})\det(A_{i+1,j}^{(k)}).\]
(By convention, we set $\det(A_{ij}^{(0)})=1$ and $\det(A_{ij}^{(-1)})=0$ for all integers $i$ and $j$.)
\end{Prop}

We visualize this relation by placing the values $\det(A_{ij}^{(k)})$ into a three-dimensional array. In all figures containing such an array, we place the entries $\det(A_{ij}^{(k)})$ at height $k$ so that, for $k>1$, $\det(A_{ij}^{(k)})$ lies directly above the center of the submatrix at height $1$ for which it is the determinant. See Figure~\ref{fig:octahedron}.






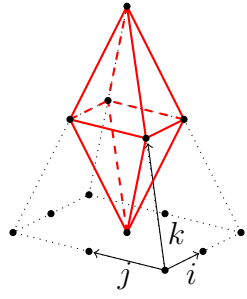
\begin{figure}
\begin{tikzpicture}[scale = 0.5]

\draw [dotted] (0,0)--(-4,1)--(-2,2)--(2,1)--cycle;
\draw [dotted] (-0.5,3.5)--(-2.5,4)--(-1.5,4.5)--(0.5,4)--cycle;

\draw [dotted] (0,0)--(-1,7);
\draw [dotted] (-4,1)--(-1,7);
\draw [dotted] (-2,2)--(-1,7);
\draw [dotted] (2,1)--(-1,7);

\draw [->] (0,0)--(-1.88,0.47);
\draw [->] (0,0)--(0.9,0.45);
\draw [->] (0,0)--(-0.45,3.36);

\node at (0.7,-0.1) {$i$};
\node at (-1.05,-0.15) {$j$};
\node at (0.3,1) {$k$};

\draw [red, thick] (-1,1)--(-2.5,4)--(-1,7)--(0.5,4)--cycle;
\draw [red, thick] (-1,1)--(-0.5,3.5)--(-1,7);
\draw [red, thick] (-2.5,4)--(-0.5,3.5)--(0.5,4);
\draw [red, dashed, thick](-1,1)--(-1.5,4.5)--(-1,7);
\draw [red, dashed, thick] (-2.5,4)--(-1.5,4.5)--(0.5,4);

\draw [fill=black] (0,0) circle [radius=0.08cm];
\draw [fill=black] (-2,0.5) circle [radius=0.08cm];
\draw [fill=black] (-4,1) circle [radius=0.08cm];

\draw [fill=black] (1,0.5) circle [radius=0.08cm];
\draw [fill=black] (-1,1) circle [radius=0.08cm];
\draw [fill=black] (-3,1.5) circle [radius=0.08cm];

\draw [fill=black] (2,1) circle [radius=0.08cm];
\draw [fill=black] (0,1.5) circle [radius=0.08cm];
\draw [fill=black] (-2,2) circle [radius=0.08cm];

\draw [fill=black] (-0.5,3.5) circle [radius=0.08cm];
\draw [fill=black] (-2.5,4) circle [radius=0.08cm];

\draw [fill=black] (0.5,4) circle [radius=0.08cm];
\draw [fill=black] (-1.5,4.5) circle [radius=0.08cm];

\draw [fill=black] (-1,7) circle [radius=0.08cm];

\end{tikzpicture}
\caption{The coordinate system that we will use for all figures depicting three-dimensional arrays $(M_{ij}^{(k)})$. The octahedron recurrence involves the vertices of translations of the octahedron shown.
}
\label{fig:octahedron}
\end{figure}

Proposition~\ref{prop:dodgson} then implies that the entries $M_{ij}^{(k)} = \det(A_{ij}^{(k)})$ satisfy the following \emph{octahedron recurrence}:
\[M_{ij}^{(k)}M_{i+1,j+1}^{(k)} = M_{i,j+1}^{(k)}M_{i+1,j}^{(k)} + M_{ij}^{(k+1)}M_{i+1,j+1}^{(k-1)}.\]
The elements in this relation lie at the vertices of an octahedron that is a translation of the one shown in Figure~\ref{fig:octahedron}. By convention, we extend the array $M=(M_{ij}^{(k)})$ to all integers $i$, $j$, and $k$ by setting any undefined values equal to $0$, which does not violate the octahedron recurrence.




We now demonstrate the known relationship between toggles and the octahedron recurrence appearing in  \cite{grinbergroby2}. A visualization of this lemma is shown in Figure~\ref{fig:toggle}. If part of the rectangle poset is labeled via the quotients $z_{ij}^{(k)}$ as shown, then Lemma~\ref{lemma:arraytoggle} shows that a toggle at $(i,j)$ transforms $z_{ij}^{(k)}$ to $z_{ij}^{(k+1)}$.

\begin{Lemma}\label{lemma:arraytoggle}
Suppose $M = (M_{ij}^{(k)})$ is an array of indeterminates satisfying the octahedron recurrence, and let $z_{ij}^{(k)} = \frac{M_{k+2,i+k+1}^{(j-1)}}{M_{k+1,i+k+1}^{(j)}}$.
Then \[z_{ij}^{(k)} = \frac{(z_{i,j-1}^{(k)}+z_{i-1,j}^{(k)})(z_{i+1,j}^{(k-1)}\parallel z_{i,j+1}^{(k-1)})}{z_{ij}^{(k-1)}}.\]


\end{Lemma}

\begin{figure}
\begin{tikzpicture}[scale = 0.6]

\draw [red, thick] (-2,2)--(-3.5,5)--(-2,8)--(-0.5,5)--cycle;
\draw [red, thick] (-2,2)--(-1.5,4.5)--(-2,8);
\draw [red, thick] (-3.5,5)--(-1.5,4.5)--(-0.5,5);
\draw [red, dashed, thick](-2,2)--(-2.5,5.5)--(-2,8);
\draw [red, dashed, thick] (-3.5,5)--(-2.5,5.5)--(-0.5,5);

\draw [blue, thick] (-3.5,5)--(-5,8)--(-3.5,11)--(-2,8)--cycle;
\draw [blue, thick] (-3.5,5)--(-3,7.5)--(-3.5,11);
\draw [blue, thick] (-5,8)--(-3,7.5)--(-2,8);
\draw [blue, dashed, thick](-3.5,5)--(-4,8.5)--(-3.5,11);
\draw [blue, dashed, thick] (-5,8)--(-4,8.5)--(-2,8);

\draw [green, thick] (-1.5,4.5)--(-3,7.5);
\draw [black, thick] (-2.5,5.5)--(-4,8.5);

\draw [fill=black] (-2,2) circle [radius=0.08cm];

\draw [fill=black] (-1.5,4.5) circle [radius=0.08cm];
\draw [fill=black] (-3.5,5) circle [radius=0.08cm];

\draw [fill=black] (-0.5,5) circle [radius=0.08cm];
\draw [fill=black] (-2.5,5.5) circle [radius=0.08cm];

\draw [fill=black] (-3,7.5) circle [radius=0.08cm];
\draw [fill=black] (-5,8) circle [radius=0.08cm];

\draw [fill=black] (-2,8) circle [radius=0.08cm];
\draw [fill=black] (-4,8.5) circle [radius=0.08cm];

\draw [fill=black] (-3.5,11) circle [radius=0.08cm];

\node at (-1.2,4.3) {$x$};
\node at (-2.2,5.2) {$x'$};

\node at (-2.6,7.3) {$y$};
\node at (-3.55,8.8) {$y'$};

\node at (-2.4,2) {$a$};
\node at (-0.1,5) {$b$};
\node at (-4,5) {$c$};
\node at (-1.6,8) {$d$};
\node at (-5.4,8) {$e$};
\node at (-3.1,11) {$f$};

\end{tikzpicture}
\begin{tikzpicture}
\transparent{0.0};
\draw (0,-2.7)--(0,2.7);
\transparent{1.0};

\draw (-1,-1)--(1,1);
\draw (-1,1)--(1,-1);

\draw [fill=black] (-1,-1) circle [radius=0.1cm];
\draw [fill=black] (1,-1) circle [radius=0.1cm];

\draw [fill=black] (0,0) circle [radius=0.1cm];

\draw [fill=black] (-1,1) circle [radius=0.1cm];
\draw [fill=black] (1,1) circle [radius=0.1cm];

\node at (-2.05,-1) {$z_{i,j-1}^{(k)} = \frac{a}{c}$};
\node at (2.05,-1) {$z_{i-1,j}^{(k)} = \frac{b}{d}$};

\node at (2.25,0) {$z_{i,j}^{(k-1)} = \frac{x}{y}\to z_{i,j}^{(k)} = \frac{x'}{y'}$};

\node at (-2.05,1) {$z_{i+1,j}^{(k-1)} = \frac{c}{e}$};
\node at (2.05,1) {$z_{i,j+1}^{(k-1)} = \frac{d}{f}$};

\end{tikzpicture}
\caption{The relationship between the octahedron recurrence and birational toggles as described in Lemma~\ref{lemma:arraytoggle}. Two labelings of the poset by quotients of entries in the octahedron recurrence are related by a birational toggle at the central element.}
\label{fig:toggle}
\end{figure}
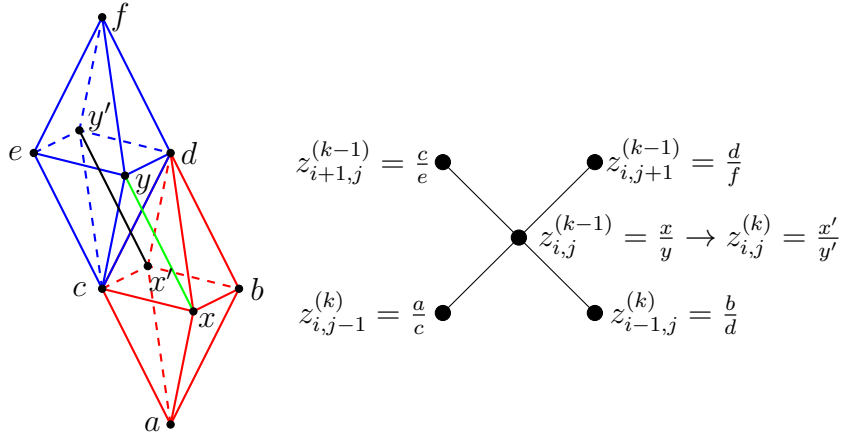

\begin{proof} For ease of notation, we relabel the relevant part of the array as in Figure \ref{fig:toggle}. Applying the octahedron recurrence to the two octahedra shown and dividing gives
\begin{align*}
    \frac{xx'}{yy'} &= \frac{ad+bc}{cf+de} 
= \frac{\frac{a}{c} + \frac{b}{d}}{\frac{f}{d} + \frac{e}{c}}
    = \left( \frac{a}{c} + \frac{b}{d} \right) \left(\frac{d}{f} \parallel  \frac{c}{e} \right).  
\end{align*}
Solving for $\frac{x'}{y'}$ yields the desired identity
\[\frac{x'}{y'}=  \frac{\left( \frac{a}{c} + \frac{b}{d} \right) \left(\frac{c}{e} \parallel  \frac{d}{f} \right)}{\frac xy}.\qedhere\]
\end{proof}

In the proof above, if we set $e = M_{k,i+k+1}^{(j)} = 0$ (which would make $\frac ce = z_{i+1,j}^{(k-1)}$ undefined), then we instead have
\begin{align*}
    \frac{xx'}{yy'} &= \frac{ad+bc}{cf} 
= \frac{\frac{a}{c} + \frac{b}{d}}{\frac{f}{d}}
    = \left( \frac{a}{c} + \frac{b}{d} \right) \left(\frac{d}{f} \right),  
\end{align*}
which yields the same result but with the undefined term $z_{i+1,j}^{(k-1)}$ removed from the parallel sum.
A similar result holds if we instead set $f = M_{k,i+k}^{(j+1)}=0$ (which would make $\frac{d}{f} = z_{i,j+1}^{(k-1)}$ undefined).

Lemma~\ref{lemma:arraytoggle} suggests that toggling/rowmotion should be thought of as a translation in the octahedron recurrence. In the next section, we will see how this relationship can be used to prove the iterated birational rowmotion formula for rectangles.

\section{Birational Rowmotion in the Rectangle Poset} \label{section:rowmotiononrectangles}

In this section, we will use the relationship between the octahedron recurrence and toggles to prove a formula for any power of birational rowmotion on a rectangle similar to one given by Musiker-Roby \cite{musikerroby}. This formula will be described in terms of nonintersecting paths inside a certain graph $\mathcal G_R$. We will begin by proving a lemma that relates nonintersecting paths in $\mathcal G_R$ to nonintersecting paths in $R$.

\subsection{Nonintersecting paths} \label{sec:paths}
Let $R=[r] \times [s]$, and let the element $(i,j)$ have weight $x_{ij}$. For $1 \leq k \leq s$, define $\mathcal P^{(k)}_R$ to be the set of all collections $\mathcal L$ of $k$ nonintersecting (i.e., vertex disjoint) paths starting at $\{(1,1), (1,2), \dots, (1,k)\}$ and ending at $\{(r,s-k+1),\dots,(r,s)\}$. The \emph{weight} $w(\mathcal L)$ is the product of the weights of the elements in the paths of $\mathcal L$. We will write $w_R^{(k)} = w_R^{(k)}(x)$ for the sum of the weights of all collections of paths in $\mathcal P^{(k)}_R$.  We also write $w_R = w_R^{(s)}$ for the product of all weights in $R$.   Similarly, if $I$ is any interval in $R$, then we define $\mathcal P_I^{(k)}$, $w_I^{(k)}(x)$, and $w_I$ in an analogous manner.

Note that while the definition of $\mathcal P_R^{(k)}$ is asymmetric in $r$ and $s$, if $k \leq \min\{r,s\}$, then the paths in any $\mathcal L \in \mathcal P_R^{(k)}$ must pass through all $k$ elements at rank $k+1$ as well as those at rank $r+s-k+1$. Thus in this case $w_R^{(k)}$ will remain unchanged if we transpose $R$. (In addition, it is simple to verify that $w_R^{(k)} = w_R$ if $k \geq \min\{r,s\}$.)

\medskip

Define $\mathcal{G}_R$ to be the graph with vertex set $[r+1] \times [s]$, directed edges from $(i,j)$ to $(i+1,j)$ of weight $x_{ij}^{-1}$, and directed edges from $(i,j)$ to $(i+1,j-1)$ of weight 1. (The vertices of $R$ are in bijection with the edges of $\mathcal{G}_R$ whose weights are not $1$.)
As an example, Figure~\ref{fig:poset_and_graph} shows $R = [2] \times [3]$ and the corresponding graph $\mathcal G_R$.

We label the vertices on the boundary of $\mathcal G_R$ as follows: for $1 \leq i \leq r$ and $1 \leq j \leq s$, define
\[P_j = (1, j), \quad P_{s+i} = (i+1, s), \quad Q_i = (i,1), \quad Q_{r+j} = (r+1, j).\]
We then define $\mathcal S^{(k)}_{ij}$ to be the set of all collections $\mathcal L'$ of $k$ nonintersecting paths starting from $\{P_i, \dots, P_{i+k-1}\}$ and ending at $\{Q_j, \dots, Q_{j+k-1}\}$. The \emph{weight} $w(\mathcal L')$ of $\mathcal L'$ is the product of the weights of the edges in the paths of $\mathcal L'$. Let $W_{ij}^{(k)} = W_{ij}^{(k)}(x)$ be the total weight of all $\mathcal L' \in \mathcal S^{(k)}_{ij}$.


\begin{figure}
\begin{tikzpicture}[scale = 1.3, rotate = 135]
\transparent{0.0};
\draw [blue] (-0.7,0.2)--(0.5,-1);
\draw [blue] (0.5,-1)--(-0.5,-2);
\transparent{1.0};

\draw (0,0)--(1,0);
\draw [red,thick] (0,-1)--(1,-1);
\draw (0,-2)--(1,-2);

\draw [thick, red] (0,0)--(0,-1);
\draw (0,-1)--(0,-2);
\draw (1,0)--(1,-1);
\draw [thick, red] (1,-1)--(1,-2);

\draw [red, fill=red] (0,0) circle [radius=0.08cm];
\draw [fill=black] (1,0) circle [radius=0.08cm];

\draw [red, fill=red] (0,-1) circle [radius=0.08cm];
\draw [red, fill=red] (1,-1) circle [radius=0.08cm];

\draw [fill=black] (0,-2) circle [radius=0.08cm];
\draw [red, fill=red] (1,-2) circle [radius=0.08cm];

\begin{scriptsize}
\node at (-0.3,-0.25) {$x_{11}$};
\node at (-0.3,-1.25) {$x_{12}$};
\node at (-0.3,-2.25) {$x_{13}$};

\node at (0.7,-0.25) {$x_{21}$};
\node at (0.7,-1.25) {$x_{22}$};
\node at (0.7,-2.25) {$x_{23}$};
\end{scriptsize}

\end{tikzpicture}
\qquad
\begin{tikzpicture}[scale = 1.3, rotate = 135]
\draw (0,0)--(1,0);
\draw [blue, thick] (1,0)--(2,0);
\draw (0,-1)--(2,-1);
\draw [blue, thick] (0,-2)--(1,-2);
\draw (1,-2)--(2,-2);

\draw [blue, thick] (0,-1)--(1,0);
\draw (0,-2)--(2,0);
\draw [blue, thick] (1,-2)--(2,-1);

\draw [fill=black] (0,0) circle [radius=0.08cm];
\draw [blue, fill=blue] (1,0) circle [radius=0.08cm];
\draw [blue, fill=blue] (2,0) circle [radius=0.08cm];

\draw [blue, fill=blue] (0,-1) circle [radius=0.08cm];
\draw [fill=black] (1,-1) circle [radius=0.08cm];
\draw [blue, fill=blue] (2,-1) circle [radius=0.08cm];

\draw [blue, fill=blue] (0,-2) circle [radius=0.08cm];
\draw [blue, fill=blue] (1,-2) circle [radius=0.08cm];
\draw [fill=black] (2,-2) circle [radius=0.08cm];

\begin{scriptsize}
\node at (-0.2,-0.2) {$P_1$};
\node at (-0.2,-1.2) {$P_2$};
\node at (-0.2,-2.2) {$P_3$};
\node at (0.8,-2.2) {$P_4$};
\node at (1.8,-2.2) {$P_5$};

\node at (0.2,0.2) {$Q_1$};
\node at (1.2,0.2) {$Q_2$};
\node at (2.2,0.2) {$Q_3$};
\node at (2.2,-0.8) {$Q_4$};
\node at (2.2,-1.8) {$Q_5$};

\node at (0.35,-0.23) {$x_{11}^{-1}$};
\node at (1.35,-0.23) {$x_{21}^{-1}$};

\node at (0.35,-1.23) {$x_{12}^{-1}$};
\node at (1.35,-1.23) {$x_{22}^{-1}$};

\node at (0.35,-2.23) {$x_{13}^{-1}$};
\node at (1.35,-2.23) {$x_{23}^{-1}$};
\end{scriptsize}
\end{tikzpicture}
\caption{The poset $R$ and graph $\mathcal G_R$ for $r=2$ and $s=3$. Edges are directed up and to the left in $\mathcal{G}_{R}$, and unlabeled edges have weight $1$. The highlighted paths correspond via the bijection in Lemma~\ref{pathsinideal}.}
\label{fig:poset_and_graph}
\end{figure}
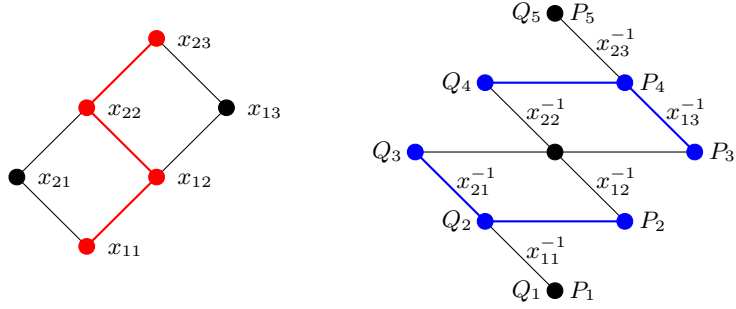

\medskip

The following result relates paths in $R$ to paths in $\mathcal{G}_R$.

\begin{Lemma} \label{lemma:gr}
There exists a bijection from $\mathcal P_R^{(k)}$ to $\mathcal S_{k+1,r+1}^{(s-k)}$ that divides weight by $w_R$.
\end{Lemma}





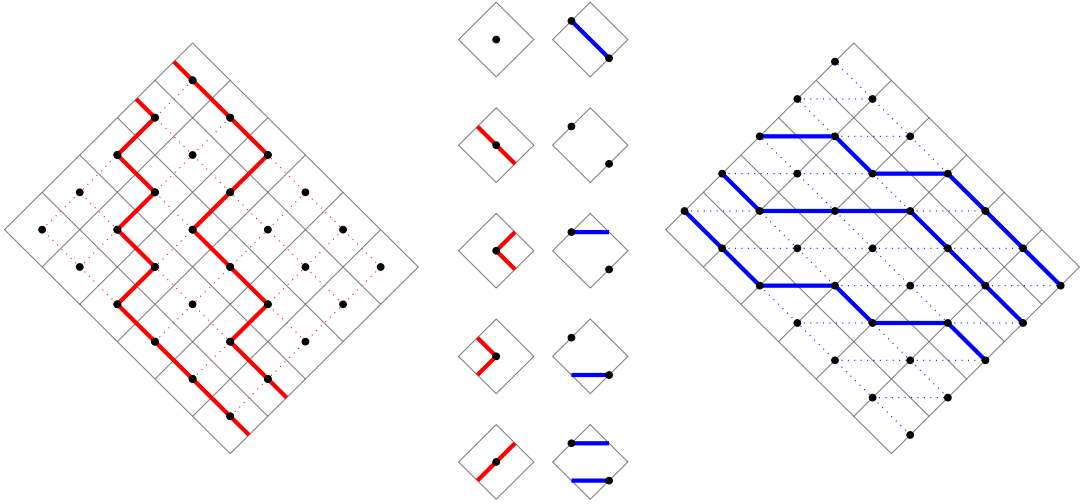
\begin{figure}
    \centering
    \begin{tikzpicture}
        \tikzstyle{w}=[circle, draw, fill=black, inner sep=0pt, minimum width=2.5pt]
        \begin{scope}[rotate=45,scale=.7]
        \draw[dotted,red] (0,0) grid (4,5);
        \draw[gray,thin,shift={(-.5,-.5)}] (0,0) grid (5,6);
        \draw[red, ultra thick] (0,-.5)--(0,3)--(1,3)--(1,4)--(2,4)--(2,5)--(3,5)--(3,5.5);
        \draw[red, ultra thick] (1,-.5)--(1,1)--(2,1)--(2,3)--(4,3)--(4,5.5);
        \foreach \x in {0,...,4}
            \foreach \y in {0,...,5}
                \node(\x\y)[w] at (\x,\y){};
        \foreach \x/\y in {0/0,0/1,0/2,0/3,1/3,1/4,2/4,2/5,3/5,1/0,1/1,2/1,2/2,2/3,3/3,4/3,4/4,4/5}
            \node[w] at (\x,\y){};
        \end{scope}
        
        \begin{scope}[shift={(3.5,4.5)},scale=.495]
        \draw[gray, thin] (0,0)--(1,1)--(0,2)--(-1,1)--(0,0);
        \node[w] at (0,1){};
        \draw[gray, thin] (2.5,0)--(3.5,1)--(2.5,2)--(1.5,1)--(2.5,0);
        \draw[blue, ultra thick] (3,.5)--(2,1.5);
        \node[w] at (3,.5){};
        \node[w] at (2,1.5){};
        \end{scope}
        \begin{scope}[shift={(3.5,3.1)},scale=.495]
        \draw[gray, thin] (0,0)--(1,1)--(0,2)--(-1,1)--(0,0);
        \draw[red, ultra thick] (.5,.5)--(-.5,1.5);
        \node[w] at (0,1){};
        \draw[gray, thin] (2.5,0)--(3.5,1)--(2.5,2)--(1.5,1)--(2.5,0);
        \node[w] at (3,.5){};
        \node[w] at (2,1.5){};
        \end{scope}
        \begin{scope}[shift={(3.5,1.7)},scale=.495]
        \draw[gray, thin] (0,0)--(1,1)--(0,2)--(-1,1)--(0,0);
        \draw[red, ultra thick] (.5,.5)--(0,1)--(.5,1.5);
        \node[w] at (0,1){};
        \draw[gray, thin] (2.5,0)--(3.5,1)--(2.5,2)--(1.5,1)--(2.5,0);
        \draw[blue, ultra thick] (3,1.5)--(2,1.5);
        \node[w] at (3,.5){};
        \node[w] at (2,1.5){};
        \end{scope}
        \begin{scope}[shift={(3.5,0.3)},scale=.495]
        \draw[gray, thin] (0,0)--(1,1)--(0,2)--(-1,1)--(0,0);
        \draw[red, ultra thick] (-.5,.5)--(0,1)--(-.5,1.5);
        \node[w] at (0,1){};
        \draw[gray, thin] (2.5,0)--(3.5,1)--(2.5,2)--(1.5,1)--(2.5,0);
        \draw[blue, ultra thick] (3,.5)--(2,.5);
        \node[w] at (3,.5){};
        \node[w] at (2,1.5){};
        \end{scope}
        \begin{scope}[shift={(3.5,-1.1)},scale=.495]
        \draw[gray, thin] (0,0)--(1,1)--(0,2)--(-1,1)--(0,0);
        \draw[red, ultra thick] (-.5,.5)--(0,1)--(.5,1.5);
        \node[w] at (0,1){};
        \draw[gray, thin] (2.5,0)--(3.5,1)--(2.5,2)--(1.5,1)--(2.5,0);
        \draw[blue, ultra thick] (3,1.5)--(2,1.5) (3,.5)--(2,.5);
        \node[w] at (3,.5){};
        \node[w] at (2,1.5){};
        \end{scope}
        
        \begin{scope}[shift={(8.7,0)},rotate=45,scale=.7]
        \draw[gray,thin,shift={(-.5,-.5)}] (0,0) grid (5,6);
        \foreach \x in {0,...,4}
            \draw[dotted,blue] (\x,-.5)--(\x,5.5);
        \foreach \x in {0,...,3}
            \foreach \y in {1,...,6}
                \draw[dotted,blue] (\x,\y-.5)--(\x+1,\y-1.5);
        \draw[blue, ultra thick] (2,-.5)--(2,.5)--(1,1.5)--(1,2.5)--(0,3.5)--(0,5.5);
        \draw[blue, ultra thick] (3,-.5)--(3,2.5)--(1,4.5)--(1,5.5);
        \draw[blue, ultra thick] (4,-.5)--(4,2.5)--(3,3.5)--(3,4.5)--(2,5.5);
        \foreach \x in {0,...,4}
            \foreach \y in {0,...,6}
                \node(\x\y)[w] at (\x,\y-.5){};
        \end{scope}
    \end{tikzpicture}
    \caption{Bijection between $\mathcal P_R^{(k)}$ and $\mathcal S_{k+1,r+1}^{(s-k)}$ as in Lemma~\ref{pathsinideal}. Tiles on the left are replaced with the corresponding tiles on the right.}
    \label{fig:5-vertex}
\end{figure}

\begin{proof}
In the drawing of $\mathcal L \in \mathcal P^{(k)}_R$, extend the start and end of each path by a half-edge from northwest to southeast. Such a drawing can be built by piecing together square tiles centered at each vertex showing the five ways in which the paths can pass through that vertex in such a way that the tiles match along sides---see Figure~\ref{fig:5-vertex}. We can replace each tile with a new tile containing edges/half-edges parallel to the edges of $\mathcal G_R$ in a way that preserves the exits along the southwest and northeast sides but inverts the exits along the southeast and northwest sides as shown in Figure~\ref{fig:5-vertex}. Placing vertices on the southeast and northwest edges of these new tiles, the resulting drawing will then depict a collection of paths $\mathcal L'$ in $\mathcal G_R$ that forms an element of $\mathcal S_{k+1,r+1}^{(s-k)}$.

The weighted edges in $\mathcal L'$ correspond to elements of $R$ not lying in any path of $\mathcal L$. It follows that $w(\mathcal L')$ is the inverse of the weight of the complement of $\mathcal L$ in $R$, which is exactly $\frac{w(\mathcal L)}{w_R}$.
\end{proof}

Let $A=(a_{ij})_{i,j=1}^{r+s}$ be the square matrix such that $a_{ij} = W_{ij}^{(1)}$, the total weight of all paths from $P_i$ to $Q_j$ in $\mathcal G_R$. By the Lindstr\"om-Gessel-Viennot Lemma  \cite{gesselviennot,lindstrom}, each minor $\det(A_{ij}^{(k)})$ is equal to $W_{ij}^{(k)}$, the total weight of $\mathcal S_{ij}^{(k)}$, the set of all collections of nonintersecting paths from $\{P_i, P_{i+1}, \dots, P_{i+k-1}\}$ to $\{Q_j, Q_{j+1}, \dots, Q_{j+k-1}\}$. By the discussion in Section~\ref{sec:octahedron}, the three-dimensional array $W=(W_{ij}^{(k)})$ formed by these minors satisfies the octahedron recurrence. See Figure~\ref{fig:array} for a depiction of $W$ when $R=[2] \times [3]$.

We will need a few simple properties of the array $W$.

\begin{Prop}\label{prop:wproperties}
Let $R=[r]\times [s]$, and let $W = (W_{ij}^{(k)})$ be the corresponding three-dimensional array. Assume $k > 0$ and $0 \leq i,j \leq r+s+1-k$.
\begin{enumerate}[(a)]
    \item If $W_{ij}^{(k)} \neq 0$, then $i \leq j \leq i+r$.
    \item If $i = j$, then $W_{ij}^{(k)}=1$.
    \item If $i < j$ and $k > s$, then $W_{ij}^{(k)}=0$.
\end{enumerate}
\end{Prop}
\begin{proof}
All of these follow from the description of $W_{ij}^{(k)}$ as the total weight of $\mathcal S_{ij}^{(k)}$. For (a), there are no paths from $P_i$ to $Q_j$ if $i > j$ or $j-i>r$. For (b), if $i=j$, then $\mathcal S_{ij}^{(k)}$ has only one element consisting only of horizontal steps of weight $1$.

For (c), if $i < j$ and $k > s$, then the first $s+1$ starting points $P_i, \dots, P_{i+s}$ all lie at or below row $i+1$, while the ending points for these paths lie at or above row $j \geq i+1$. But row $i+1$ only has $s$ vertices, so it is impossible for $s+1$ nonintersecting paths to pass through it.
\end{proof}

We can translate Lemma~\ref{lemma:gr} into a statement expressing weights of certain collections of nonintersecting paths in $R$ in terms of the array of minors $W=(W_{ij}^{(k)})$. In fact, we can formulate a similar result for paths inside certain intervals $I \subseteq R$.

\begin{Cor}\label{pathsinideal}
Let $R = [r] \times [s]$, and let $I=[i_1,i_2] \times [j_1,j_2]$ be an interval in $R$ such that $i_1=1$ or $j_2=s$, and $j_1=1$ or $i_2=r$.
\begin{enumerate}[(a)]
    \item 
    There exists a bijection from $\mathcal P_I^{(k)}$ to $\mathcal S_{i_1+j_1+k-1,i_2+j_1}^{(j_2-j_1-k+1)}$ that divides weight by $w_I$, so that
    \[W_{i_1+j_1+k-1,i_2+j_1}^{(j_2-j_1-k+1)} = \frac{w_I^{(k)}}{w_I}.\]
    \item The following equality holds:
\[w_I^{(k)} 
= \frac{W_{i_1+j_1+k-1,i_2+j_1}^{(j_2-j_1-k+1)}}{W_{i_1+j_1-1,i_2+j_1}^{(j_2-j_1+1)}}.\]
\end{enumerate}
\end{Cor}
\begin{proof}
Any nonintersecting paths in $\mathcal G_R$ starting at $P_{i_1+j_1+k-1}, \dots, P_{i_1+j_2-1}$ must begin with horizontal steps until reaching row $i_1$ at points $(i_1, j_1+k), \dots, (i_1,j_2)$. Similarly, any nonintersecting paths ending at $Q_{i_2+j_1}, \dots, Q_{i_2+j_2-k}$ must pass through row $i_2+1$ at $(i_2+1, j_1), \dots, (i_2+1, j_2-k)$ and end with horizontal steps. The remaining parts of the paths lie inside the subgraph of $\mathcal G_R$ from rows $i_1$ through $i_2+1$ and columns $j_1$ through $j_2$, which is isomorphic as a weighted graph to $\mathcal G_I$. Part (a) then follows by applying Lemma~\ref{lemma:gr} to $I$ and summing over all paths.

For part (b), 
setting $k=0$ in part (a) gives $W^{(j_2-j_1+1)}_{i_1+j_1-1,i_2+j_1} = \frac{1}{w_I}$. Combining  with part (a) gives the result.
\end{proof}

In particular, note that the conditions of Corollary~\ref{pathsinideal} hold whenever $I$ is any order ideal or order filter of $R$.

\subsection{Birational rowmotion formula}

Using the relation between the octahedron recurrence and toggles, we will prove the following birational rowmotion formula.

\begin{Th}
\label{birationalrowmotionformula}
Let $R = [r] \times [s]$, $x \in \mathbb{R}_{>0}^{R}$, and $y = \phi^{-1}(x)$. 
Fix $(i,j) \in R$.



\begin{enumerate}[(a)]
\item If $0 \leq k \leq r+s-i-j$, then
\[\rho^{-k}(y)_{ij} = \frac{W_{k+2,i+k+1}^{(j-1)}}{W_{k+1,i+k+1}^{(j)}}.\]
\item For all $k \in \mathbb{Z}$,
\[\rho^{k}(y)_{ij} = \frac{1}{\rho^{k-i-j+1}(y)_{r+1-i,s+1-j}}.\]
In particular, if $0 < k < i+j$, then the right hand side can be computed using part (a).
\item For all $k \in \mathbb Z$, $\rho^{k+r+s}(y) = \rho^k(y)$. In other words, the action of rowmotion on $R$ has order $r+s$.
\end{enumerate}
\end{Th}

\begin{figure}
\begin{tikzpicture}[scale = 0.9]

\begin{scope}[scale=1.2,shift={(-6,7.5)},rotate=135]
\draw (0,0)--(1,0);
\draw (0,-1)--(1,-1);
\draw (0,-2)--(1,-2);

\draw (0,0)--(0,-2);
\draw (1,0)--(1,-2);

\draw [fill=black] (0,0) circle [radius=0.08cm];
\draw [fill=black] (1,0) circle [radius=0.08cm];

\draw [fill=black] (0,-1) circle [radius=0.08cm];
\draw [fill=black] (1,-1) circle [radius=0.08cm];

\draw [fill=black] (0,-2) circle [radius=0.08cm];
\draw [fill=black] (1,-2) circle [radius=0.08cm];

\begin{small}
\node at (-0.2,-0.2) {$a$};
\node at (-0.2,-1.2) {$c$};
\node at (-0.2,-2.2) {$e$};

\node at (0.8,-0.2) {$b$};
\node at (0.8,-1.2) {$d$};
\node at (0.8,-2.2) {$f$};
\end{small}
\end{scope}

\begin{scope}[scale = 1.2, shift={(3,7.0)}, rotate = 135]
\draw (0,0)--(1,0);
\draw (1,0)--(2,0);
\draw (0,-1)--(2,-1);
\draw (0,-2)--(1,-2);
\draw (1,-2)--(2,-2);

\draw (0,-1)--(1,0);
\draw (0,-2)--(2,0);
\draw (1,-2)--(2,-1);

\draw [fill=black] (0,0) circle [radius=0.08cm];
\draw [fill=black] (1,0) circle [radius=0.08cm];
\draw [fill=black] (2,0) circle [radius=0.08cm];

\draw [fill=black] (0,-1) circle [radius=0.08cm];
\draw [fill=black] (1,-1) circle [radius=0.08cm];
\draw [fill=black] (2,-1) circle [radius=0.08cm];

\draw [fill=black] (0,-2) circle [radius=0.08cm];
\draw [fill=black] (1,-2) circle [radius=0.08cm];
\draw [fill=black] (2,-2) circle [radius=0.08cm];

\begin{scriptsize}
\node at (-0.2,-0.2) {$P_1$};
\node at (-0.2,-1.2) {$P_2$};
\node at (-0.2,-2.2) {$P_3$};
\node at (0.8,-2.2) {$P_4$};
\node at (1.8,-2.2) {$P_5$};

\node at (0.2,0.2) {$Q_1$};
\node at (1.2,0.2) {$Q_2$};
\node at (2.2,0.2) {$Q_3$};
\node at (2.2,-0.8) {$Q_4$};
\node at (2.2,-1.8) {$Q_5$};

\node at (0.39,-0.18) {$\overline{a}$};
\node at (1.39,-0.18) {$\overline{b}$};

\node at (0.39,-1.18) {$\overline{c}$};
\node at (1.39,-1.18) {$\overline{d}$};

\node at (0.39,-2.18) {$\overline{e}$};
\node at (1.39,-2.21) {$\overline{f}$};
\end{scriptsize}
\end{scope}

\draw [red,thick] (-1.625,-0.25)--(-2,0.5);
\draw [red,thick] (-3.625,0.25)--(-4,1);

\draw [red,thick] (-1,1)--(-2.5,4);
\draw [red,thick] (-3,1.5)--(-4.5,4.5);

\draw [red,thick] (-1.5,4.5)--(-3,7.5);
\draw [red,thick] (-3.5,5)--(-5,8);

\draw [blue,thick] (-2.625,0.75)--(-3,1.5);
\draw [blue,thick] (-4.625,1.25)--(-5,2);

\draw [blue,thick] (-2,2)--(-3.5,5);
\draw [blue,thick] (-4,2.5)--(-5.5,5.5);

\draw [blue,thick] (-2.5,5.5)--(-4,8.5);

\draw [fill=black] (0,0) circle [radius=0.08cm];
\draw [fill=black] (-2,0.5) circle [radius=0.08cm];
\draw [fill=black] (-4,1) circle [radius=0.08cm];
\draw [fill=black] (-6,1.5) circle [radius=0.08cm];
\draw [fill=black] (-8,2) circle [radius=0.08cm];

\draw [fill=black] (1,0.5) circle [radius=0.08cm];
\draw [fill=black] (-1,1) circle [radius=0.08cm];
\draw [fill=black] (-3,1.5) circle [radius=0.08cm];
\draw [fill=black] (-5,2) circle [radius=0.08cm];
\draw [fill=black] (-7,2.5) circle [radius=0.08cm];

\draw [fill=black] (2,1) circle [radius=0.08cm];
\draw [fill=black] (0,1.5) circle [radius=0.08cm];
\draw [fill=black] (-2,2) circle [radius=0.08cm];
\draw [fill=black] (-4,2.5) circle [radius=0.08cm];
\draw [fill=black] (-6,3) circle [radius=0.08cm];

\draw [fill=black] (3,1.5) circle [radius=0.08cm];
\draw [fill=black] (1,2) circle [radius=0.08cm];
\draw [fill=black] (-1,2.5) circle [radius=0.08cm];
\draw [fill=black] (-3,3) circle [radius=0.08cm];
\draw [fill=black] (-5,3.5) circle [radius=0.08cm];

\draw [fill=black] (4,2) circle [radius=0.08cm];
\draw [fill=black] (2,2.5) circle [radius=0.08cm];
\draw [fill=black] (0,3) circle [radius=0.08cm];
\draw [fill=black] (-2,3.5) circle [radius=0.08cm];
\draw [fill=black] (-4,4) circle [radius=0.08cm];

\draw [fill=black] (-0.5,3.5) circle [radius=0.08cm];
\draw [fill=black] (-2.5,4) circle [radius=0.08cm];
\draw [fill=black] (-4.5,4.5) circle [radius=0.08cm];
\draw [fill=black] (-6.5,5) circle [radius=0.08cm];

\draw [fill=black] (0.5,4) circle [radius=0.08cm];
\draw [fill=black] (-1.5,4.5) circle [radius=0.08cm];
\draw [fill=black] (-3.5,5) circle [radius=0.08cm];
\draw [fill=black] (-5.5,5.5) circle [radius=0.08cm];

\draw [fill=black] (1.5,4.5) circle [radius=0.08cm];
\draw [fill=black] (-0.5,5) circle [radius=0.08cm];
\draw [fill=black] (-2.5,5.5) circle [radius=0.08cm];
\draw [fill=black] (-4.5,6) circle [radius=0.08cm];

\draw [fill=black] (2.5,5) circle [radius=0.08cm];
\draw [fill=black] (0.5,5.5) circle [radius=0.08cm];
\draw [fill=black] (-1.5,6) circle [radius=0.08cm];
\draw [fill=black] (-3.5,6.5) circle [radius=0.08cm];

\draw [fill=black] (-1,7) circle [radius=0.08cm];
\draw [fill=black] (-3,7.5) circle [radius=0.08cm];
\draw [fill=black] (-5,8) circle [radius=0.08cm];

\draw [fill=black] (0,7.5) circle [radius=0.08cm];
\draw [fill=black] (-2,8) circle [radius=0.08cm];
\draw [fill=black] (-4,8.5) circle [radius=0.08cm];

\draw [fill=black] (1,8) circle [radius=0.08cm];
\draw [fill=black] (-1,8.5) circle [radius=0.08cm];
\draw [fill=black] (-3,9) circle [radius=0.08cm];

\draw [fill=black] (-1.5,10.5) circle [radius=0.08cm];
\draw [fill=black] (-3.5,11) circle [radius=0.08cm];

\draw [fill=black] (-2.5,11.5) circle [radius=0.08cm];
\draw [fill=black] (-0.5,11) circle [radius=0.08cm];

\draw [fill=black] (-2,14) circle [radius=0.08cm];

\draw [dotted] (0,0)--(-8,2)--(-4,4)--(4,2)--cycle;
\draw [dotted] (-0.5,3.5)--(-6.5,5)--(-3.5,6.5)--(2.5,5)--cycle;
\draw [dotted] (-1,7)--(-5,8)--(-3,9)--(1,8)--cycle;
\draw [dotted] (-1.5,10.5)--(-3.5,11)--(-2.5,11.5)--(-0.5,11)--cycle;


\draw [dotted] (0,0)--(-2,14);
\draw [dotted] (-8,2)--(-2,14);
\draw [dotted] (-4,4)--(-2,14);
\draw [dotted] (4,2)--(-2,14);

\begin{footnotesize}
\node at (0.4,0) {$1$};
\node at (-1.5,0.5) {$\overline{a}$};
\node at (-3.5,1) {$\overline{ab}$};
\node at (-5.6,1.5) {$0$};
\node at (-7.6,2) {$0$};

\node at (1.3,0.5) {$0$};
\node at (-0.7,1) {$1$};
\node at (-2.4,1.5) {$\overline{b} + \overline{c}$};
\node at (-4.6,2) {$\overline{cd}$};
\node at (-6.6,2.5) {$0$};

\node at (2.3,1) {$0$};
\node at (0.3,1.5) {$0$};
\node at (-1.7,2) {$1$};
\node at (-3.4,2.5) {$\overline{d} + \overline{e}$};
\node at (-5.6,3) {$\overline{ef}$};

\node at (3.3,1.5) {$0$};
\node at (1.3,2) {$0$};
\node at (-0.7,2.5) {$0$};
\node at (-2.7,3) {$1$};
\node at (-4.7,3.5) {$\overline{f}$};

\node at (4.3,2) {$0$};
\node at (2.3,2.5) {$0$};
\node at (0.3,3) {$0$};
\node at (-1.7,3.5) {$0$};
\node at (-3.7,4) {$1$};

\node at (-0.2,3.5) {$1$};
\node at (-2.1,4) {$\overline{ac}$};
\node at (-4.0,4.5) {$\overline{abcd}$};
\node at (-6.1,5) {$0$};

\node at (0.8,4) {$0$};
\node at (-1.2,4.5) {$1$};
\node at (-2.3,5) {$\overline{bd} + \overline{be} + \overline{ce}$};
\node at (-4.9,5.5) {$\overline{cdef}$};

\node at (1.8,4.5) {$0$};
\node at (-0.2,5) {$0$};
\node at (-2.2,5.5) {$1$};
\node at (-4.1,6) {$\overline{df}$};

\node at (2.8,5) {$0$};
\node at (0.8,5.5) {$0$};
\node at (-1.1,6) {$0$};
\node at (-3.2,6.5) {$1$};

\node at (-0.7,7) {$1$};
\node at (-2.6,7.5) {$\overline{ace}$};
\node at (-4.3,8) {$\overline{abcdef}$};

\node at (0.3,7.5) {$0$};
\node at (-1.7,8) {$1$};
\node at (-3.6,8.5) {$\overline{bdf}$};

\node at (1.3,8) {$0$};
\node at (-0.7,8.5) {$0$};
\node at (-2.7,9) {$1$};

\node at (-1.2,10.5) {$1$};
\node at (-3.2,11) {$0$};

\node at (-0.2,11) {$0$};
\node at (-2.2,11.5) {$1$};

\node at (-1.7,14) {$1$};

\node at (-0.4,-0.2) {$Q_1$};
\node at (-2.4,0.3) {$Q_2$};
\node at (-4.4,0.8) {$Q_3$};
\node at (-6.4,1.3) {$Q_4$};
\node at (-8.4,1.8) {$Q_5$};

\node at (0.6,-0.3) {$P_1$};
\node at (1.6,0.2) {$P_2$};
\node at (2.6,0.7) {$P_3$};
\node at (3.6,1.2) {$P_4$};
\node at (4.6,1.7) {$P_5$};
\end{footnotesize}

\end{tikzpicture}
\caption{$R=[2] \times [3]$, $\mathcal G_R$, and the corresponding array $W=W_{ij}^{(k)}$ for $1 \leq k \leq 5$, where $\overline{a}$ denotes $a^{-1}$ for readability. (All entries at height $0$ equal $1$, and all other entries not shown are $0$.) Red and blue lines indicate quotients used to compute $\rho^0(y)=y=\phi^{-1}(x)$ and $\rho^{-1}(y)$ (apart from the topmost label), respectively.}
\label{fig:array}
\end{figure}

Theorem~\ref{birationalrowmotionformula} can be used to find an explicit formula for the entries of any power of $\rho$ applied to $y$. (Using parts (a) and (b) one can compute $\rho^{k}(y)_{ij}$ whenever $i+j-r-s \leq k < i+j$, and part (c) can be used to bring $k$ into this range.) While the birational rowmotion formulas given in \cite{musikerroby} are stated in terms of complements of paths in $R$, the formulation given here can be seen to be equivalent using Corollary~\ref{pathsinideal}.

Geometrically, Theorem~\ref{birationalrowmotionformula}(a) states that one can find the values of $\rho^{-k}(y)_{ij}$ as quotients of nearby entries inside the array $W$, and increasing the value of $k$ corresponds to a translation in the direction $(1,1,0)$ as long as these entries remain inside the defined pyramidal region of $W$. Upon passing outside this region, one needs to use part (b) to relocate inside the pyramid to the entries corresponding to the antipodal point of $R$. See Figure~\ref{fig:array}.

\begin{Ex}
Let $R = [2] \times [3]$. Let $x$ be the labeling of $R$ as in Figure~\ref{fig:array} with the array $(W_{ij}^{(k)})$ as previously described. When $(i,j)=(2,2)$, Theorem \ref{birationalrowmotionformula}(a) gives
\begin{align*}
    \rho^{0}(y)_{22} &= \frac{W_{23}^{(1)}}{W_{13}^{(2)}} =  \frac{b^{-1} + c^{-1}}{(abcd)^{-1}} = acd + abd,\\
    \rho^{-1}(y)_{22} &= \frac{W_{34}^{(1)}}{W_{24}^{(2)}} = \frac{d^{-1} + e^{-1}}{(cdef)^{-1}} = cef + cdf.
\end{align*}

To compute $\rho^{-2}(y)_{22}$, $r+s-i-j = 1 < 2$, so part (a) does not apply. Instead, we apply parts (c) and (b) first to obtain
\begin{align*}
    \rho^{-2}(y)_{22} &= \rho^{3}(y)_{22} = \frac{1}{\rho^{0}(y)_{12}} = \frac{W_{12}^{(2)}}{W_{22}^{(1)}}\\
    &= (ac)^{-1}, \\
    \rho^{-3}(y)_{22} &= \rho^2(y)_{22} = \frac{1}{\rho^{-1}(y)_{12}} = \frac{W_{23}^{(2)}}{W_{33}^{(1)}} \\&= (bd)^{-1}+(be)^{-1}+(ce)^{-1},\\
    \rho^{-4}(y)_{22} &= \rho(y)_{22} = \frac{1}{\rho^{-2}(y)_{12}} =
    \frac{W_{34}^{(2)}}{W_{44}^{(1)}} \\&= (df)^{-1}.
\end{align*}
\end{Ex}

We will prove each part of Theorem~\ref{birationalrowmotionformula} separately. Part (a) follows primarily from Lemma~\ref{lemma:arraytoggle}, though some care is needed along the boundary of $R$.

\begin{proof}[Proof of Theorem~\ref{birationalrowmotionformula}(a)]
Let $z_{ij}^{(k)} = \frac{W_{k+2,i+k+1}^{(j-1)}}{W_{k+1,i+k+1}^{(j)}}$, which we wish to equal $\rho^{-k}(y)_{ij}$ for $0 \leq k \leq r+s-i-j$. We proceed by induction on $k$. When $k = 0$, $\rho^0(y)_{ij} = y_{ij} = \phi^{-1}(x)_{ij}$ is the total weight of all maximal chains in $[i] \times [j]$ (with respect to the labeling $x$). By Corollary~\ref{pathsinideal}(b), this is $\frac{W_{2,i+1}^{(j-1)}}{W_{1,i+1}^{(j)}} = z_{ij}^{(0)}$, as desired.

For the inductive step, assume $k > 0$ and suppose that the claim is true for $\rho^{-k+1}(y)$ as well as for $\rho^{-k}(y)_{(i',j')}$ if $(i',j') \prec (i,j)$. The value of $\rho^{-k}(y)$ is obtained by applying toggles from bottom to top on $\rho^{-k+1}(y)$, so the toggle at $(i,j)$ gives
\[\rho^{-k}(y)_{ij} = \left(\sum_{(i',j')\lessdot (i,j)} \rho^{-k}(y)_{i'j'}\right)\left(\sideset{}{^\parallel}\sum_{(i',j') \gtrdot (i,j)}\rho^{-k+1}(y)_{i'j'}\right) \cdot \frac{1}{\rho^{-k+1}(y)_{ij}}, \tag{$*$} \label{eq:toggle}\]
where an empty sum is replaced with $1$. 

We claim that we can replace the first sum in \eqref{eq:toggle} with $z_{i,j-1}^{(k)}+z_{i-1,j}^{(k)}$. If $i > 1$ and $j > 1$, then this is immediate by induction. Otherwise note that
\[
    z_{i0}^{(k)} = \frac{W_{k+2,i+k+1}^{(-1)}}{W_{k+1,i+k+1}^{(0)}} = \frac01= 0,\qquad
    z_{0j}^{(k)} = \frac{W_{k+2,k+1}^{(j-1)}}{W_{k+1,k+1}^{(j)}} = \frac{\delta_{j1}}{1} = \delta_{j1}
\]
by Proposition~\ref{prop:wproperties}(a) and (b).
Thus if exactly one of $i$ and $j$ is $1$, then one of $z_{i,j-1}^{(k)}$ and $z_{i-1,j}^{(k)}$ equals $0$ while the other is the only term in the first sum in \eqref{eq:toggle} by induction. If instead $i=j=1$, then the empty sum in \eqref{eq:toggle} is replaced with $1=0 + \delta_{11} = z_{10}^{(k1)} + z_{01}^{(k)}$ as needed.

Similarly, we claim that we can replace the second (parallel) sum in \eqref{eq:toggle} with either $z_{i,j+1}^{(k-1)} \parallel z_{i+1,j}^{(k-1)}$, or just one of these terms if the other is undefined. If $i < r$ and $j < s$, then this is again immediate by induction (since the inequality $k-1 \leq r+s-(i+j+1)$ holds). If $i=r$ and $j < s$, then $z_{r+1,j}^{(k-1)}$ is undefined since $W_{k,r+k+1}^{(j)}=0$ by Proposition~\ref{prop:wproperties}(a), so we are left with $z_{r,j+1}^{(k-1)} = \rho^{-k+1}(y)_{r,j+1}$ by induction. If instead $i < r$ and $j=s$, then $z_{i,s+1}^{(k-1)}$ is undefined since $W_{k,i+k}^{(s+1)} = 0$ by Proposition~\ref{prop:wproperties}(c), which leaves just $z_{i+1,s}^{(k-1)} = \rho^{-k+1}(y)_{i+1,s}$ by induction. Finally, $i=r$ and $j=s$ cannot both occur since $0 < k \leq r+s-i-j$.

The final factor in \eqref{eq:toggle} is equal to $\frac{1}{z^{(k-1)}_{ij}}$ by induction. The result then follows by comparing \eqref{eq:toggle} to Lemma~\ref{lemma:arraytoggle}.
\end{proof}

Next we prove part (b). In some sense, this reflects two symmetries: toggling is respected by dualizing a poset and inverting each label, and the octahedron recurrence is respected by reflecting over the plane perpendicular to the main diagonal.

\begin{proof}[Proof of Theorem~\ref{birationalrowmotionformula}(b)]
Note that if $0 < k < i+j$, then $0 \leq i+j-k-1 \leq r+s-(r+1-i)-(s+1-j)$, so the right hand side of (b) will satisfy the condition of part (a), as claimed.

Via the change of coordinates $z = \rho^{k-1}(y)$, it suffices to prove the case when $k=1$. 
By part (a),
\[\rho^{2-i-j}(y)_{r+1-i,s+1-j} = 
    \frac{W_{i+j,r+j}^{(s-j)}}{W_{i+j-1,r+j}^{(s-j+1)}}.\]
By Corollary~\ref{pathsinideal}, this is the total weight of all maximal chains in $[i,r] \times [j,s]$ in $R$ (with respect to the labeling $x$). Thus
\[\rho^{2-i-j}(y)_{r+1-i,s+1-j} = \left(\phi^*\right)^{-1}(x)_{ij} = \frac{1}{\rho(y)_{ij}}\]
by Lemma~\ref{lemma:dualtransfer}.
\end{proof}











%
%

Finally, we deduce part (c) from part (b).
%

\begin{proof} [Proof of Theorem~\ref{birationalrowmotionformula}(c)]
Apply Theorem~\ref{birationalrowmotionformula}(b) twice:
\begin{align*}
    \rho^{k}(y)_{i,j} &= \left(\rho^{k-i-j+1}(y)_{r+1-i,s+1-j}\right)^{-1}\\ 
    &= \left[\left(\rho^{k-i-j+1-(r+1-i)-(s+1-j)+1}(y)_{i,j}\right)^{-1}\right]^{-1}\\
    & = \rho^{k-r-s}(y)_{i,j}.\qedhere
\end{align*}
\end{proof}

In the next two sections, we will show how to use the perspective relating promotion and the octahedron recurrence to prove various identities about rowmotion.


\section{A Chain Shifting Lemma} \label{section:chainshifting}

In this section we will use the birational rowmotion formula to derive a chain shifting lemma, which is a birational generalization of the rowmotion action on the Stanley-Thomas word appearing in \cite{josephroby1}. 

\subsection{Chain Shifting}

Recall from Section~\ref{sec:paths} that if $I$ is an interval of $R$ and $x$ is a labeling of $R$, then $w_I^{(k)}(x)$ is the total weight of the collections of $k$ nonintersecting paths in $\mathcal P_I^{(k)}$. We will prove the following lemma relating these sums in two labelings $x$ and $\phi \circ \rho^{-1} \circ \phi^{-1}(x)$.


\begin{Lemma}[Chain Shifting] \label{chainshift}
Let $R = [r] \times [s]$ and let $1 < u \leq v \leq s$. Define intervals $I = [r] \times [u,v]$ and $I' = [r] \times [u-1,v-1]$ of $R$. Then for any labeling $x \in \mathbb R^{R}_{>0}$, \[w_{I}^{(k)}(x) = w_{I'}^{(k)}(\phi \circ \rho^{-1} \circ \phi^{-1}(x)).\]
(The symmetric statement obtained by reflecting $R$ also holds.)
%
%
\end{Lemma}





Note that the intervals $I$ and $I'$ differ only by shifting by one unit in $R$.

\begin{Ex}

\begin{figure}
\begin{tikzpicture}[scale = 0.7]

\foreach \x in {0,10.4}{
\draw (0+\x,0)--(2+\x,2);
\draw (-1+\x,1)--(1+\x,3);

\draw (0+\x,0)--(-1+\x,1);
\draw (1+\x,1)--(0+\x,2);
\draw (2+\x,2)--(1+\x,3);

\draw [fill=black] (0+\x,0) circle [radius=0.1cm];
\draw [fill=black] (1+\x,1) circle [radius=0.1cm];
\draw [fill=black] (2+\x,2) circle [radius=0.1cm];

\draw [fill=black] (-1+\x,1) circle [radius=0.1cm];
\draw [fill=black] (0+\x,2) circle [radius=0.1cm];
\draw [fill=black] (1+\x,3) circle [radius=0.1cm];
};

\node at (0.35,0) {$a$};
\node at (1.45,1) {$c$};
\node at (2.3,2) {$e$};

\node at (-0.6,1) {$b$};
\node at (0.4,2) {$d$};
\node at (1.4,3) {$f$};

\draw [xshift=4.25 cm, yshift=1.5cm] (0,0.1)--(1,0.1)--(1,0.2)--(1.5,0)--(1,-0.2)--(1,-0.1)--(0,-0.1)--cycle;

\node at (5.25,2.1) {$\phi \circ \rho^{-1} \circ \phi^{-1}$};

\node at (11,0) {$\frac{bc}{b+c}$};
\node at (12.5,1) {$\frac{de(b+c)}{bd+cd+ce}$};
\node at (13.8,2) {$\frac{f(bd+cd+ce)}{ce}$};

\node at (8.6,1) {$\frac{d(b+c)}{b}$};
\node at (8.95,2) {$\frac{f(bd+cd+ce)}{d(b+c)}$};
\node at (9.8,3.1) {$\frac{1}{af(bd+cd+ce)}$};
\end{tikzpicture}
\caption{Two labelings $x$ and $z = \phi \circ \rho^{-1} \circ \phi^{-1}(x)$.}
\label{fig:shifting}
\end{figure}
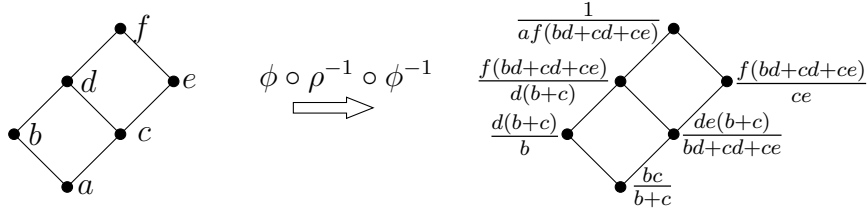

Consider the labeling $x$ of $R = [2] \times [3]$ and its image $z = \phi \circ \rho^{-1} \circ \phi^{-1}(x)$ in Figure~\ref{fig:shifting}. Then
\[z_{11}z_{21} = \frac{bc}{b+c}\cdot \frac{d(b+c)}{b} = cd = x_{12}x_{22}.\]
Similarly, we can compute all of the products along the rows and columns of $R$:
\begin{align*}
    z_{12}z_{22} &= x_{13}x_{23} & z_{13}z_{23} &= (x_{11}x_{12}x_{13})^{-1} \\
    (z_{11}z_{12}z_{13})^{-1} &= (x_{21}x_{22}x_{23})^{-1}&(z_{21}z_{22}z_{23})^{-1} &= x_{11}x_{21}
\end{align*}


From the above computation, we see that $\phi \circ \rho^{-1} \circ \phi^{-1}$ rotates the birational Stanley-Thomas word
\[(x_{11}x_{21}, \quad x_{12}x_{22}, \quad x_{13}x_{23},\quad  (x_{11}x_{12}x_{13})^{-1},\quad (x_{21}x_{22}x_{23})^{-1} ).\]

However, the shifting of chain sums happens more generally for sums of chains within subrectangles of $R$. For instance, if $I = [2] \times [2,3]$, $I' = [2] \times [1,2]$, and $k=1$, then we have
\begin{align*}
    z_{11}z_{21}z_{22}+z_{11}z_{12}z_{22} &= \frac{bc}{b+c}\cdot \frac{f(bd+cd+ce)}{d(b+c)} \left( \frac{d(b+c)}{b} + \frac{de(b+c)}{bd+cd+ce} \right) \\
    &= \frac{cf(bd+cd+ce)}{b+c} + \frac{bcef}{b+c}\\
    &= cdf+cef \\
    &= x_{12}x_{22}x_{23}+x_{12}x_{13}x_{22}.
\end{align*}
\end{Ex}

We now prove Lemma \ref{chainshift}. As we will see, it can be thought of as a manifestation of the translation invariance of the octahedron recurrence.

\begin{proof}[Proof of Lemma~\ref{chainshift}] Let $x$ be a labeling of $R$, let $y = \phi^{-1}(x)$, and let $(W_{ij}^{(k)})$ be the corresponding three-dimensional array. Also define $\widetilde x = \phi \circ \rho^{-1}\circ \phi^{-1}(x)$, $\widetilde y = \phi^{-1}(\widetilde x) = \rho^{-1}(y)$, and $(\widetilde W_{ij}^{(k)})$ the corresponding array.


We claim that $\widetilde W_{ij}^{(k)} = W_{i+1,j+1}^{(k)}$ for $1 \leq i,j \leq r+s-k$. Note that by the octahedron recurrence/Dodgson condensation formula, it suffices to prove the case $k=1$ since the entries for $k > 1$ are determinants of submatrices of these entries at height $1$. 

Assume $k=1$. If $i \geq j$, then the claim follows by Proposition~\ref{prop:wproperties}. If $i < j$, then applying Theorem~\ref{birationalrowmotionformula} twice gives
\[\widetilde W_{ij}^{(1)} = \frac{1}{\rho^{-i+1}(\widetilde y)_{j-i,1}} = \frac{1}{\rho^{-i}(y)_{j-i,1}} = W^{(1)}_{i+1,j+1},\] which proves the claim.

To complete the proof, we need only observe that applying Corollary~\ref{pathsinideal}(b) to the two sides of the desired equality expresses them as quotients of two entries of the form $W_{i+1,j+1}^{(k)}$ or $\widetilde W_{ij}^{(k)}$, respectively, which are then equal by the claim.
%
%
%
%
\end{proof}

\subsection{Generalized Stanley-Thomas words}
Given Lemma~\ref{chainshift}, it is natural to want to generalize the definition of the Stanley-Thomas word to include more words that are cyclically shifted by the action of rowmotion. 

The usual birational Stanley-Thomas word is obtained as the orbit of $\prod_{j=1}^s x_{1j} = \phi^{-1}(x)_{1s}$ under the action of $\phi \circ \rho \circ \phi^{-1}$ on $x$, or, put another way, the orbit of $y_{1s}$ under the action of $\rho$ on $y=\phi^{-1}(x)$. We can generalize this by replacing $y_{1s}$ with any other value $y_{is}$ or $y_{rj}$ lying on the upper boundary of $R$.

\begin{Def}
Let $x \in \mathbb R^R_{>0}$, and let $y=\phi^{-1}(x)$. A \emph{generalized Stanley-Thomas word} for $x$ is a sequence of one of the following two forms:
\begin{align*}
    ST_{i}(x)&=(y_{is}, \quad \rho^{-1}(y)_{is}, \quad \rho^{-2}(y)_{is}, \quad \dots, \quad \rho^{-r-s+1}(y)_{is}),\\
    \overline{ST}_j(x)&=(y_{rj}, \quad \rho^{-1}(y)_{rj}, \quad \rho^{-2}(y)_{rj}, \quad \dots, \quad \rho^{-r-s+1}(y)_{rj}).
\end{align*}
\end{Def}

The usual birational Stanley-Thomas word is therefore $ST_{1}$. The reason for choosing these words in particular is that they have a clean description in terms of $x$. We consider the case of $ST_i$ below; the other case is similar.

Recall that $y_{is} = \phi^{-1}(x)_{is} = w^{(1)}_{[i]\times [s]}(x)$. By Lemma~\ref{chainshift}, 
\[\rho^{-1}(y)_{is} = w^{(1)}_{[i] \times [s]}(\phi \circ \rho^{-1}(y)) = w^{(1)}_{[2,i+1] \times [s]}(x).\]
Iterating, we find that for $0 \leq k \leq r-i$,
\[\rho^{-k}(y)_{is} = w^{(1)}_{[k+1,k+i] \times [s]}(x).\]
To find the remaining coordinates, we apply Theorem~\ref{birationalrowmotionformula} (c), (b), and (a) in order to find that, for $r-i<k<r+s$,
\[\rho^{-k}(y)_{is} = \rho^{r+s-k}(y)_{is} = \frac{1}{\rho^{r-i-k+1}(y)_{r+1-i,1}} = W_{k+i-r,k+1}^{(1)}.\]

While $W_{k+i-r,k+1}^{(1)}$ is defined in terms of paths in $\mathcal G_R$, it is also easy to describe it in terms of the labeling $x$ directly. Choosing the northwest edges in a path in $\mathcal G_R$ from $P_{k+i-r}$ to $Q_{k+1}$ corresponds to choosing elements of $R$ at ranks $k+i-r+1,\dots, k+1$ traveling to the northwest (i.e., with increasing first coordinate), not necessarily using the edges of $R$. In other words, define
\[\omega_{ab} = \omega_{ab}(x) = \sum (x_{i_aj_a}x_{i_{a+1}j_{a+1}} \cdots x_{i_bj_b})^{-1},\]
where the sum ranges over all sequences $(i_a,j_a), (i_{a+1},j_{a+1}), \dots, (i_b,j_b) \in R$ such that $i_t+j_t = t$ for all $t$, and $i_a < i_{a+1} < \cdots < i_b$ (or equivalently $j_a \geq j_{a+1} \geq \cdots \geq j_b$). Then we have proved the following proposition.

\begin{Prop} \label{prop:st}
For any $x \in \mathbb R^R_{>0}$, the generalized Stanley-Thomas word $ST_i = ST_i(x)$ is given by
\begin{multline*}
    ST_i = (w^{(1)}_{[1,i] \times [s]}, \; w^{(1)}_{[2,i+1] \times [s]}, \; \dots, \; w^{(1)}_{[k+1,k+i] \times [s]}, \;
    \omega_{2,r-i+2}, \; \omega_{3,r-i+3}, \; \dots, \; \omega_{i+s,r+s}).
\end{multline*}
\end{Prop}
An analogous formula for $\overline{ST}_j$ can be obtained by reflecting $R$.

\begin{Ex} \label{ex:st}
Let $R = [3] \times [3]$ and $x \in \mathbb R^R_{>0}$. The generalized Stanley-Thomas word $ST_2$ is 
\[\ST_2 = (w^{(1)}_{[1,2] \times [3]}, \;w^{(1)}_{[2,3] \times [3]}, \;\omega_{23},\; \omega_{34},\; \omega_{45},\; \omega_{56}).\]
In the $x$-coordinates we have the following:
\begin{align*}
    w^{(1)}_{[1,2] \times [3]} &= x_{11}x_{21}x_{22}x_{23}+x_{11}x_{12}x_{22}x_{23}+x_{11}x_{12}x_{13}x_{23},
    \\
    w^{(1)}_{[2,3] \times [3]} &= x_{21}x_{31}x_{32}x_{33}+x_{21}x_{22}x_{32}x_{33}+x_{21}x_{22}x_{23}x_{33},\\
    \omega_{23} &= \frac{1}{x_{11}x_{21}}, \\
    \omega_{34} &= \frac{1}{x_{12}x_{22}} + \frac{1}{x_{12}x_{31}} + \frac{1}{x_{21}x_{31}}, \\
    \omega_{45} &= \frac{1}{x_{13}x_{23}} + \frac{1}{x_{13}x_{32}} + \frac{1}{x_{22}x_{32}}, \\
    \omega_{56} &= \frac{1}{x_{23}x_{33}}.
\end{align*}
See Figure~\ref{fig:st} for a visualization of this example.
\end{Ex}

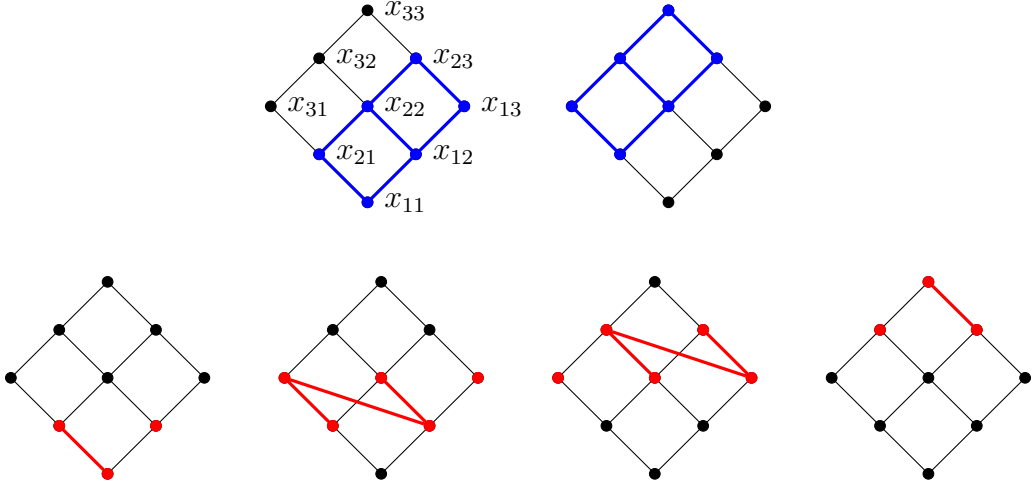
\begin{figure}
    \centering
    \begin{tikzpicture}[scale=.9]
        \begin{scope} [shift={(3.8,4)},rotate=45]
            \draw (1,1) grid (3,3);
            \foreach \x in {1,...,3}
                \foreach \y in {1,...,3}
                    \node[w,label=right:{$x_{\x\y}$}] (\x\y) at (\y,\x){};
            \draw[very thick,blue] (1,1) grid (3,2);
            \foreach \z in {11,12,13,21,22,23}
                \node[w,blue] at (\z){};
        \end{scope}
        \begin{scope} [shift={(8.2,4)},rotate=45]
            \draw (1,1) grid (3,3);
            \foreach \x in {1,...,3}
                \foreach \y in {1,...,3}
                    \node[w] (\x\y) at (\y,\x){};
            \draw[very thick,blue] (1,2) grid (3,3);
            \foreach \z in {21,22,23,31,32,33}
                \node[w,blue] at (\z){};
        \end{scope}
        \begin{scope} [rotate=45]
            \draw (1,1) grid (3,3);
            \foreach \x in {1,...,3}
                \foreach \y in {1,...,3}
                    \node[w] (\x\y) at (\y,\x){};
            \draw[very thick,red] (11)--(21);
            \foreach \z in {11,21,12}
                \node[w,red] at (\z){};
        \end{scope}
        \begin{scope} [shift={(4,0)},rotate=45]
            \draw (1,1) grid (3,3);
            \foreach \x in {1,...,3}
                \foreach \y in {1,...,3}
                    \node[w] (\x\y) at (\y,\x){};
            \draw[very thick,red] (22)--(12)--(31)--(21);
            \foreach \z in {22,12,31,21,13}
                \node[w,red] at (\z){};
        \end{scope}
        \begin{scope} [shift={(8,0)},rotate=45]
            \draw (1,1) grid (3,3);
            \foreach \x in {1,...,3}
                \foreach \y in {1,...,3}
                    \node[w] (\x\y) at (\y,\x){};
            \draw[very thick,red] (23)--(13)--(32)--(22);
            \foreach \z in {23,13,32,22,31}
                \node[w,red] at (\z){};
        \end{scope}
        \begin{scope} [shift={(12,0)},rotate=45]
            \draw (1,1) grid (3,3);
            \foreach \x in {1,...,3}
                \foreach \y in {1,...,3}
                    \node[w] (\x\y) at (\y,\x){};
            \draw[very thick,red] (23)--(33);
            \foreach \z in {23,33,32}
                \node[w,red] at (\z){};
        \end{scope}
    \end{tikzpicture}
    \caption{Computation of $ST_2$ in $R=[3] \times [3]$ as in Example~\ref{ex:st}. The first two diagrams indicate the total weight of maximal chains in intervals that shift according to Lemma~\ref{chainshift}. The last four diagrams indicate the total inverse weight of northwesterly collections of elements at specified ranks used to find $\omega_{ab}$.}
    \label{fig:st}
\end{figure}

As noted in \cite{josephroby1}, the cyclic rotation of $ST_1$ is not sufficient to uniquely determine birational rowmotion. However, the cyclic rotation of all $ST_i$ and $\overline{ST}_j$ does uniquely determine birational rowmotion, and in fact the chain shifting lemma alone nearly suffices. We make this statement precise in Section \ref{section:greenestheorem}.

\section{Birational RSK and Greene's Theorem} \label{section:greenestheorem}

In this section, we will define birational RSK in terms of toggles and show how our perspective gives a simple proof of the birational version of Greene's Theorem.

\subsection{Classical RSK} \label{sec:classical-rsk}

We first review some background on the classical RSK correspondence, which gives a bijection between nonnegative integer matrices $A$ and pairs of semistandard tableaux $(P,Q)$ of the same shape $\lambda$. (See, for instance, \cite{sagan} for more details.)

In the case when $A$ is the matrix of a permutation $\pi$, Greene's Theorem \cite{greene} states that $\lambda_1 + \cdots + \lambda_k$ is the maximum size of a union of $k$ increasing subsequences of $\pi$. In fact, one can use Greene's Theorem to compute not just the shape of $P$ and $Q$ but the entire tableaux: the shape $P^{\leq m}$ formed by the entries at most $m$ in $P$ corresponds to the permutation formed by the letters $1, \dots, m$ in $\pi$, while the shape $Q^{\leq m}$ corresponds to the permutation formed by the first $m$ letters of $\pi$.

When $A$ is a general $n \times n$ nonnegative integer matrix, a routine standardization argument gives the following generalization of Greene's Theorem: $\lambda_1 + \cdots + \lambda_k$ is the maximum weight of $k$ noncrossing paths (traveling weakly southeast) from $(1,1), \dots, (1,k)$ to $(n,n-k+1), \dots, (n,n)$ in $A$, where the weight of a path is the sum of the entries it contains. (As noted in \cite{farberhopkinstrongsiriwat}, this result appears to be somewhat folklore, but see \cite[Thm.\ 2.5]{noumiyamada} as well as \cite[Thm.\ 12]{krattenthaler}, \cite[Thm.\ 4.8.10]{sagan}.)
As in the permutation case, this can be used to characterize the entire $P$ and $Q$ tableaux since $P^{\leq m}$ (resp.\ $Q^{\leq m}$) corresponds to the submatrix formed by the first $m$ columns (resp.\ rows) of $A$. 

One way to see the impact of Greene's Theorem more directly is to transform $P$ and $Q$ into Gelfand-Tsetlin patterns and glue them along their top rows to form an $n \times n$ matrix with weakly increasing rows and columns that we denote by $RSK(A)$. By construction, this matrix has the property that, for $1 \leq k \leq i,j \leq n$ with either $i=n$ or $j=n$, 
\[\sum_{t=0}^{k-1} RSK(A)_{i-t,j-t}\] is the maximum weight of $k$ noncrossing paths from $(1,1), \dots, (1,k)$ to $(i,j-k+1), \dots, (i,j)$ in $A$.

In the next section, we will see how the map $A \mapsto RSK(A)$ can be generalized to the birational setting.

\subsection{Birational RSK}

For any interval $I \subseteq R$, denote by $\rho_I$ the map on labelings of $R$ given by the composition of toggles at the elements of $I$ (applied in the order of a linear extension from top to bottom).

\begin{Def}
Let $R = [r] \times [s]$, and let $m = \min\{r,s\}-1$. The \emph{birational RSK map} on labelings of $R$ is given by
\[RSK = \rho_{[r-m]\times[s-m]}^{-1} \circ \dots \circ \rho_{[r-2] \times [s-2]}^{-1} \circ \rho_{[r-1] \times [s-1]}^{-1} \circ \phi^{-1}.\]
\end{Def}

We will compare this to other formulations of birational RSK below, but for now, this can be treated as a definition. We will abuse notation and denote the piecewise-linear version of this map by $RSK$ as well.

Although $RSK$ is defined using toggles on portions of $R$, one can also describe the entries of $RSK(x)$ using the action of rowmotion on all of $R$.

\begin{Prop}
$RSK(x)_{r-i,s-j} = \rho^{-\min\{i,j\}} \circ \phi^{-1}(x)_{r-i,s-j}.$
\label{Cor:RSKasRowmotion}
\end{Prop}

\begin{proof} 
We claim that $\rho_{[r-k]\times[s-k]}^{-1} \circ \cdots \circ \rho^{-1}_{[r-1]\times[s-1]}(y)$ and $\rho^{-k} (y)$ agree on $[r-k]\times[s-k]$. Indeed, suppose inductively that the claim holds for $k-1$. Since any neighbor of an element in $[r-k] \times [s-k]$ lies in $[r-k+1] \times [s-k+1]$, applying $\rho^{-1}_{[r-k,s-k]}$ to both $\rho_{[r-k+1]\times[s-k+1]}^{-1} \circ \cdots \circ \rho^{-1}_{[r-1]\times[s-1]}(y)$ and $\rho^{-k+1} (y)$ yields labelings that agree on $[r-k]\times[s-k]$. Then applying the remaining toggles in $\rho^{-1}$ (at elements of $R \setminus ([r-k] \times [s-k])$) to the latter also does not affect any of these values, proving the claim.

If $\min\{i,j\} = k$, then $\rho_{[r-k']\times[s-k']}^{-1}$ does not affect the label at $(r-i,s-j)$ for any $k'>k$, so the result follows from the claim.
\end{proof}

We are now ready to prove the following theorem, which will serve as a birational analogue to Greene's Theorem as described in Section~\ref{sec:classical-rsk}.

\begin{Th}[Birational Greene's Theorem] \label{plGreene}
Let $R = [r] \times [s]$, and choose $(i,j) \in R$ such that either $i = r$ or $j = s$. Then for any $x \in \mathbb R^{R}_{>0}$ and $1\leq k \leq \min \{i,j\}+1$,
\[\prod_{t=0}^{k-1} RSK(x)_{i-t,j-t} = w_{[i]\times[j]}^{(k)}(x).\]
\end{Th}

\begin{proof}
Let $x \in \mathbb{R}_{>0}^{R}$ and let $y = \phi^{-1}(x)$. By Proposition \ref{Cor:RSKasRowmotion}, since \[\min\{r-(i-t),s-(j-t)\} = \min\{r-i,s-j\}+t = t,\] 
we have
\[RSK(x)_{i-t,j-t} = \rho^{-t}(y)_{i-t,j-t}.\] Theorem~\ref{birationalrowmotionformula}(a) then gives
\[\rho^{-t}(y)_{i-t,j-t} = \frac{W_{t+2,i+1}^{(j-t-1)}}{W_{t+1,i+1}^{(j-t)}}.\]
We then have the telescoping product
\[\prod_{t=0}^{k-1} RSK(x)_{i-t,j-t}=\prod_{t=0}^{k-1} \rho^{-t}(y)_{i-t,j-t} = \frac{W_{k+1,i+1}^{(j-k)}}{W_{1,i+1}^{(j)}} = w_{[i]\times[j]}^{(k)}(x)\]
by Corollary~\ref{pathsinideal}.
%
%
%
%
\end{proof}

Note that Theorem~\ref{plGreene} uniquely characterizes the map $RSK$ in terms of the values of $w_{[i]\times[j]}^{(k)}(x)$ where $i=r$ or $j=s$. To see why this can be considered to be a birational version of Greene's Theorem, consider the tropicalization of Theorem~\ref{plGreene}:
\[\sum_{t=0}^{k-1} RSK(x)_{i-t,j-t} = \max_{\mathcal{L} \in \mathcal{P}_{[i]\times[j]}^{(k)}} \sum_{z \in \mathcal{L}} x_z.\]
This corresponds directly to the description of $RSK(A)$ given in Section~\ref{sec:classical-rsk}. We can therefore treat this as a proof that the map $RSK$ as defined above is a birational analogue to the classical RSK correspondence.

\subsection{Relation to prior work}

Our construction of birational RSK using the octahedron recurrence is very similar to the approach taken by Danilov-Koshevoy \cite{danilovkoshevoy} (although their construction differs from the standard one by a symmetry of $A$). Our construction is also necessarily equivalent to the description of ``tropical RSK'' given by Noumi-Yamada \cite{noumiyamada} since their map is also constructed to satisfy Theorem~\ref{plGreene}.

In \cite{farberhopkinstrongsiriwat}, a birational version of RSK is described using the following piecewise-linear analogue of RSK defined in \cite{hopkins,pak}.
Let $x \in \mathcal{C}(R)$ be a point in the chain polytope of $R=[r]\times[s]$. We construct $y = rsk(x) \in \mathcal O(R)$ in the order polytope of $R$ using the following procedure.

\begin{enumerate}
    \item Set $y_{11} = x_{11}$.
    \item Choose an element $p \in R$ such that $y_q$ has been defined for all $q < p$. Set $y_p = x_p + \max\limits_{q \lessdot p} y_q$, and then toggle $y$ at all elements below $p$ in the same file.
    \item Repeat the previous step until all coordinates of $y$ have been defined.
\end{enumerate}

In fact, this map is the same as the piecewise-linear version of the map $RSK$ defined in the previous section.

\begin{Prop}
$RSK=rsk$.
\end{Prop}

\begin{proof}
When performing step (2) above, none of the toggled elements form a cover relation with any element that has yet to be considered. Since the value of each new $y_p$ depends only on the values $y_q$ for $q \lessdot p$, the toggles at the previous steps do not affect the initial value of $y_p$. Hence we can apply all the toggles after first assigning all $y_p$, which is equivalent to first applying $\phi^{-1}$ to $x$.

Without loss of generality, we may assume $r \leq s$. If $r = 1$, then the rectangle is a chain and no toggles are applied, so $rsk = \phi^{-1} = RSK$.

Suppose inductively that on $[r]\times[s]$, $rsk = \rho_{[1] \times [s-r+1]}^{-1} \circ \dots \circ \rho_{[r-1]\times[s-1]}^{-1} \circ \phi^{-1}$, and consider the rectangle $R=[r+1] \times [s]$ for $r+1 \leq s$. After applying $rsk$ on the order ideal $[r] \times [s]$, to complete $rsk$ for $R$ we must apply toggles at all elements in $[r]\times[s]$ that lie in files $1-r, 2-r,\dots, s-1-r$. Since toggles at elements that are not adjacent commute, we can toggle row by row instead of file by file. Write $I_i = \{i\} \times [s-r-1+i] \subseteq R$ for $i \leq r$. We then find that on $R$,
\begin{align*}
    rsk &= \left( \rho_{I_1}^{-1} \circ \rho_{I_2}^{-1} \dots \circ \rho^{-1}_{I_r}  \right) \circ \rho_{[1] \times[s-r+1]}^{-1} \circ \dots \circ \rho_{[r-1]\times[s-1]}^{-1} \circ \phi^{-1} \\
    &= \left( \rho_{I_1}^{-1} \right) \circ \left( \rho_{I_2}^{-1} \circ \rho_{[1]\times[s-r+1]}^{-1} \right) \circ \dots \circ \left( \rho_{I_r}^{-1} \circ \rho_{[r-1]\times[s-1]}^{-1} \right) \circ \phi^{-1} \\
    &= \rho_{[1]\times[s-r]}^{-1} \circ \rho_{[2]\times[s-r+1]}^{-1} \circ \dots \circ \rho_{[r]\times[s-1]}^{-1} \circ \phi^{-1}\\
    &= RSK.\qedhere
\end{align*}
\end{proof}


\subsection{Relation to chain sums and Stanley-Thomas words}





Since $RSK$ is invertible, any $x \in \mathbb R^R_{>0}$ is uniquely determined by $RSK(x)$. From Theorem~\ref{plGreene}, it follows that the values of $w_{I}^{(k)}(x)$, where $I$ has the form $[r]\times[j]$ or $[i]\times[s]$, determine $x$.

Perhaps even more relevant for our work with rowmotion, we can also characterize $x$ using the types of chain sums that appear in Lemma~\ref{chainshift}.

\begin{Th}
Let $x \in \mathbb{R}^R_{>0}$. The chain sums $w_I^{(1)}(x)$, where $I$ ranges over intervals of the form $[r] \times [u,v]$ and $[u,v] \times [s]$, uniquely determine $x$.
\end{Th}
\begin{proof}
By the Lindstr\"om-Gessel-Viennot Lemma on $R$ (using vertex weights), we can express any $w_{[r] \times [j]}^{(k)}$ as a determinant of a matrix with entries of the form $w_{[r] \times [u,v]}^{(1)}$, and similarly any $w_{[i] \times [s]}^{(k)}$  is determined by the values $w_{[u,v] \times [s]}^{(1)}$.
By Theorem~\ref{plGreene}, each entry of $RSK(x)$ is a quotient of (or equal to) entries of the form $w_{[r]\times[j]}^{(k)}$ or $w_{[i]\times[s]}^{(k)}$, so the result follows.
\end{proof}

In particular, since these chain sums all appear in the generalized Stanley-Thomas words, the following corollary is immediate.

\begin{Cor}
The map $\phi \circ \rho \circ \phi^{-1}$ is the unique function on $\mathbb R^R_{>0}$ that cyclically shifts the generalized Stanley-Thomas words $ST_i$ and $\overline{ST}_j$.
\end{Cor}

By Lemma~\ref{chainshift}, most of the chain sums $w_I^{(1)}$ considered above get sent to other such sums. It follows that the chain shifting lemma determines ``most'' of rowmotion in the following sense.


\begin{Cor}
The chain shifting property of $\phi \circ \rho^{-1} \circ \phi^{-1}$ (stated in Lemma~\ref{chainshift})
uniquely determines $RSK(\phi \circ \rho^{-1} \circ \phi^{-1}(x))_{ij}$, where $j-i \neq s-r$.
\end{Cor}

\begin{proof}
To compute the $(i,j)$th coordinate of the $RSK$ map where $j-i<s-r$, we need only compute the values of $w_{[r] \times [j+r-i]}^{(k)}$ for $k \leq r-i+1$ by Theorem~\ref{plGreene}. Since $j+r-i<s$, by Lemma~\ref{chainshift}, $w_{[r] \times [j+r-i]}^{(k)}(\phi \circ \rho^{-1} \circ \phi^{-1}(x)) = w_{[r],[2,j+r-i+1]}^{(k)}(x)$.

The entries with $j-i>s-r$ can likewise be obtained using the reflected version of Lemma~\ref{chainshift}.
\end{proof}

\section{Acknowledgments}
The authors would like to thank Sam Hopkins, Alexander Postnikov, Darij Grinberg, and Tom Roby for useful conversations.

\bibliographystyle{acm}
\bibliography{citations}

\end{document}